\documentclass[11pt]{amsart}

\usepackage{lmodern}

\usepackage[font=small,labelfont=bf]{caption}
 \usepackage{amsmath, amsthm, amssymb, amsfonts}
\usepackage[utf8]{inputenc}
\usepackage[normalem]{ulem}
\usepackage[colorlinks = true,
            linkcolor = black,
            urlcolor  = blue,
            citecolor = red,
            anchorcolor = red]{hyperref}
\usepackage{url}
\usepackage{comment}
\usepackage{mathrsfs}  
\usepackage{longtable} 
\usepackage{enumitem}
\usepackage{relsize}
\usepackage{bm}
\usepackage{xcolor}
\usepackage{graphicx}
\usepackage{mathabx}
\usepackage{mathdots}
\usepackage{mathtools}
\usepackage{color}
\usepackage{tikz-cd,tikz}
\usepackage{braket}
\usepackage{stmaryrd}
\usepackage{amssymb}
\usepackage{tipa}
\usepackage{appendix}
\usepackage{physics}

\usepackage{appendix}

\usepackage{tikz}
\usetikzlibrary{decorations.pathmorphing, arrows}
\usepackage{tikz-cd}

\hypersetup{
    colorlinks=true,
    linkcolor=black,
    filecolor=magenta,
    urlcolor=cyan,
    citecolor=magenta,
    pdftitle={Rationality stable pairs},
    pdfpagemode=FullScreen,
}

\tikzset{snake it/.style={decorate, decoration=snake}}

\theoremstyle{definition}
\newtheorem{definition}{Definition}[section]

\newtheorem{theorem}[definition]{Theorem}
\newtheorem{example}[definition]{Example}
\newtheorem{proposition}[definition]{Proposition}

\newtheorem{corollary}[definition]{Corollary}

\newtheorem{lemma}[definition]{Lemma}
\newtheorem{conjecture}[definition]{Conjecture}

\newtheorem{question}[definition]{Question}
\newtheorem{remark}[definition]{Remark}

\newtheorem{thm}{Theorem}

\newcommand{\pt}{\mathsf{pt}}

\newcommand{\BA}{\mathbb{A}}

\newcommand{\BC}{\mathbb{C}}
\newcommand{\BD}{\mathbb{D}}

\newcommand{\BG}{\mathbb{G}}

\newcommand{\BK}{\mathbb{K}}

\newcommand{\BP}{\mathbb{P}}
\newcommand{\BQ}{\mathbb{Q}}
\newcommand{\BR}{\mathbb{R}}

\newcommand{\BZ}{\mathbb{Z}}

\newcommand{\CA}{\mathcal A}
\newcommand{\CB}{\mathcal B}

\newcommand{\CE}{\mathcal E}
\newcommand{\CF}{\mathcal F}

\newcommand{\CI}{\mathcal I}

\newcommand{\CO}{\mathcal O}

\newcommand{\CS}{\mathcal S}

\newcommand{\CX}{\mathcal X}

\newcommand{\FM}{\mathfrak{M}}

\newcommand{\FX}{\mathfrak{X}}

\newcommand{\bL}{\mathsf{L}}
\newcommand{\bR}{\mathsf{R}}

\newcommand{\sM}{\mathsf{M}}
\newcommand{\sL}{\mathsf{L}}

\DeclareMathOperator{\id}{id}

\DeclareMathOperator{\Spec}{Spec}
\DeclareMathOperator{\Pic}{Pic}

\DeclareMathOperator{\Hom}{Hom}
\DeclareMathOperator{\Ext}{Ext}

\newcommand{\GW}{\mathsf{GW}}
\newcommand{\PT}{\mathsf{PT}}
\newcommand{\DT}{\mathsf{DT}}
\newcommand{\GV}{\mathsf{GV}}
\newcommand{\vir}{\mathrm{vir}}

\newcommand{\ch}{\mathrm{ch}}
\newcommand{\td}{\mathrm{td}}

\newcommand{\Coh}{\mathrm{Coh}}

\newcommand{\sym}{\text{sym}}

\newcommand{\rk}{\mathrm{rk}}

\newcommand{\rig}{\mathrm{rig}}
\newcommand{\inv}{{\mathrm{wt}_0}}

\newcommand{\VX}{\mathbf{V}_X}
\newcommand{\bV}{\mathbf{V}}

\newcommand{\VXhat}{\widecheck{\mathbf{V}}_X}

\newcommand{\bx}{{\bf x}}
\newcommand{\bn}{{\bf n}}
\newcommand{\bt}{{\bf t}}
\newcommand{\omord}{\omega_{\textup{ord}}}
\newcommand{\FL}{\mathfrak{L}}
\newcommand{\Filt}{\mathrm{Filt}}

\newcommand{\Frac}{\mathrm{Frac}}

\newcommand{\lpp}{(\hspace{-0.15em}(}
\newcommand{\rpp}{)\hspace{-0.15em})}

\begin{document}

\title[Rationality of PT generating series]{Rationality and symmetry of stable pairs generating series of Fano 3-folds}
\date{\today}

\newcommand\blfootnote[1]{%
  \begingroup
  \renewcommand\thefootnote{}\footnote{#1}%
  \addtocounter{footnote}{-1}%
  \endgroup
}

\date{\today}

\author[I. Karpov]{Ivan Karpov}
\address{Massachusetts Institute of Technology, Department of Mathematics}
\email{karpov57@mit.edu}

\author[M. Moreira]{Miguel Moreira}
\address{Massachusetts Institute of Technology, Department of Mathematics}
\email{miguel@mit.edu}


\maketitle

\setcounter{tocdepth}{1}

\begin{abstract}
The generating series of descendent invariants of stable pairs on 3-folds is conjectured to be rational and to satisfy a $q\leftrightarrow q^{-1}$ symmetry. We prove this conjecture for Fano 3-folds. We utilize the same path of stability conditions that Toda used in his proof of the Calabi--Yau version of the conjecture, relating stable pairs and $L$ invariants, and work of the two authors that allows an extension of Joyce's descendent wall-crossing formula to non-standard hearts of $D^b(X)$. We use Ehrhart theory to deal with the combinatorics coming out of the wall-crossing formula. Furthermore, we specialize the wall-crossing formula to primary insertions and prove a strong rationality result predicted by the Pandharipande--Thomas/Gopakumar--Vafa correspondence.
\end{abstract}

\tableofcontents

\section{Introduction}

Stable pairs, introduced by Pandharipande and Thomas in \cite{PT}, are one of the various curve counting theories on 3-folds. A stable pair is a pure 1-dimensional sheaf $F$ together with a section $\CO_X\to F$ whose cokernel is 0-dimensional.

Let $X$ be a smooth projective 3-fold and let $P_{\beta, m}=P_{\beta, m}(X)$ be the moduli space of stable pairs $[\CO_X\to F]$ with $\ch_2(F)=\beta$ and $\ch_3(F)=m[\pt]$. Here, $m$ is related to the Euler characteristic of $F$ by 
\[m=\chi-\frac{1}{2}\beta\cdot c_1(X)\in \frac{1}{2}\BZ\,.\]
The moduli spaces $P_{\beta, m}$ admit a virtual fundamental class 
\[[P_{\beta, m}]^\vir\in H_{2d_\beta}(P_{\beta, m})\]
with virtual dimension
\begin{equation}
    d_\beta=\int_X\beta\cdot c_1(X)\in \BZ\,.\label{eq: dbeta}
\end{equation}
The virtual fundamental class can be used to define invariants, often referred to as \textit{stable pair invariants} or \textit{Pandharipande--Thomas} ($\PT$) \textit{invariants}. When $X$ is Calabi--Yau, the virtual dimension $d_\beta$ is equal to 0, in which case the $\PT$ invariants are obtained by taking the degree of the virtual fundamental class. When that is not the case, one obtains invariants by integrating tautological classes:
\begin{equation}\label{eq: ptinvariants}
    \int_{[P_{\beta, m}]^\vir}\ch_{k_1}(\gamma_1)\ldots\ch_{k_n}(\gamma_n)\in \BQ\,.
\end{equation}
Here, 
\[\ch_k(\gamma)=p_\ast\big(\ch_k(\mathcal I)q^\ast \gamma\big)\in H^\ast(P_{\beta, m})\]
where $\mathcal I=[\CO_{P_{\beta, m}\times X}\to \CF]$ is the universal stable pair, which is a perfect complex in $P_{\beta, m}\times X$, $\gamma\in H^\ast(X)$ and $p,q$ are the projections of $P_{\beta, m}\times X$ onto $P_{\beta, m}$ and $X$, respectively. 

Given a fixed curve class $\beta$ and a formal combination of tautological classes
\begin{equation}
    \label{eq: D}
D=\ch_{k_1}(\gamma_1)\ldots\ch_{k_n}(\gamma_n)\,,\end{equation}
the $\PT$ generating series is the formal Laurent series\footnote{It is more common in the literature to define the generating function with a sum over $q^\chi$, where $\chi$ is the Euler characteristic, instead of $q^m$. These different conventions just change the generating series by a factor of $q^{d_\beta/2}$. Our convention removes the factor of $q^{-d_\beta}$ from the functional equation in \cite[Conjecture 4]{survey}.}
\[Z_\beta^\PT(q|D)\coloneqq \sum_{m}q^m\int_{[P_{\beta, m}]^\vir}D\in q^{d_\beta/2}\BQ\lpp q\rpp \,.\]

Stable pairs were introduced as a better version of the moduli space of ideal sheaves and the invariants associated to it, often referred to as Donaldson--Thomas ($\DT$) invariants. It was conjectured in \cite{mnop1} that a certain normalization of the DT generating series was rational (in the Calabi--Yau case) and matched the Gromov--Witten ($\GW$) generating series after a change of variables. Ideal sheaves can heuristically be thought of as a curve together with isolated points in $X$, and the normalization of the DT generating series is meant to formally ``throw away'' the contribution of these isolated points. On the other hand, a stable pair corresponds roughly to a curve together with points \textit{on the curve}, so in some sense stable pairs are geometrically ``throwing away'' the contributions of isolated points. Pandharipande--Thomas conjectured that the $\PT$ generating series matches the normalized $\DT$ generating series, and in particular that it is rational and satisfies a certain functional equation:\footnote{The matching of normalized $\DT$ and $\PT$ is only true on the nose in the Calabi--Yau case, in general a more complicated relation involving non-rational functions is expected \cite[Conjecture 5.2.1]{oblomkovDTPT} and \cite[Section 4.2]{OOP}. The $\PT$ generating series are still expected to be rational, while the normalized DT series are not.}

\begin{conjecture}\label{conj: rationality}
    Let $X$ be a smooth projective $3$-fold, $\beta$ an effective curve class and $D$ as in \eqref{eq: D}. Then $Z_\beta^\PT(q|D)$ is the Laurent expansion of a rational function\footnote{More precisely, the expansion of $q^{d_\beta/2}f(q)$ where $f$ is a rational function in $q$.} which satisfies the functional equation
\[Z_\beta^\PT(q^{-1}|D)=(-1)^{k_1+\ldots+k_n}Z^\PT_\beta(q |D)\,.\]
\end{conjecture}

There are several extensions of the conjecture: an equivariant version when $X$ is acted on by a torus, a version for log pairs $(X, \partial X)$, and a version for families of 3-folds $\CX\to B$; see \cite{survey, pandharipandefamilies} for two surveys on the subject. 

Conjecture \ref{conj: rationality} has been proven in many cases. When $X$ is a Calabi--Yau 3-fold -- in which case only $D=1$ is interesting -- it is known by work of Toda \cite{todaPTrationality} and Bridgeland \cite{bridgelandDTPT}. Pandharipande and Pixton proved the rationality part of the conjecture, and some cases of the functional equation part, for several geometries: local curves in \cite{PPlocalcurves}, toric 3-folds in \cite{PPtoric} and Fano complete intersections in products of projective spaces (with insertions $\gamma_i$ pulled back from the ambient space) in \cite{PPquintic}. Work of Pardon \cite{pardon} showed that the conjecture holds when $X$ is Fano for primary insertions (i.e. $k_i=2$); Pardon explains that, in some sense, the curve counting theories for Calabi--Yau 3-folds or for Fano 3-folds with primary insertions are totally controlled by what happens for local curves, but this is no longer the case with non-primary insertions. Maulik and Ranganathan \cite{MRgwpttoric} have recently proven rationality for toric log geometries and primary insertions. There is also a virtual cobordism version \cite{shencobordism} which, as explained in \cite[Section 4]{shencobordism} and \cite[Section 4]{survey}, is a formal consequence of Conjecture~\ref{conj: rationality}. A surface version of this conjecture, with insertions involving the Chern classes of the virtual tangent bundle, is proposed in \cite{pavlakovic}.

The main result of this paper is the proof of Conjecture \ref{conj: rationality} for Fano 3-folds and, more generally, for  ``strongly positive'' curve classes. Recall that if $\beta', \beta$ are two curve classes we say that $\beta'\leq \beta$ if $\beta-\beta'$ is effective.
\begin{thm}\label{thm: main}
Let $X$ be a 3-fold such that $h^{p,0}(X)=0$ for $p>0$ and let $\beta$ be a curve class such that $-\beta'\cdot K_X>0$ for any effective curve class $\beta'\leq \beta$. Then Conjecture \ref{conj: rationality} holds. 
\end{thm}

A very similar result has been independently obtained in \cite{AJ}. We discuss briefly the comparison with their result and method in Section \ref{subsec: comparison}.

\begin{remark}Both conditions of the theorem hold when $X$ is Fano. The condition on the Hodge numbers is a consequence of Kodaira vanishing.\end{remark}

The positivity condition seems absolutely crucial in our approach, at least at the moment, and is used to guarantee that certain moduli spaces (e.g. moduli of 1-dimensional sheaves on $X$, but see also Theorem \ref{thm: quasismooth}) used in the proof have virtual fundamental classes. The condition on the Hodge numbers of $X$, on the other hand, is just a simplifying assumption so that we do not need to use a traceless deformation theory or, in other words, so that we can ignore that stable pairs should be treated as a fixed determinant theory; in general, the derived tangent bundle to $P_{\beta, m}$ at a stable pair $I=[\CO_X\to F]$ is controlled by
\[\Ext^i(I, I)_0=\ker\big(\!\Ext^i(I, I)\to H^i(\CO_X)\big)\,.\]
This can likely be removed with not too much work, but since the Fano case is our main interest we will not address this question in the present paper.

\begin{remark}
By keeping track of the poles in the proof of Theorem \ref{thm: main} (see in particular Remark \ref{rmk: combinatorialpoles}) one can show that the only possible poles of $Z_{\beta}^\PT(q|D)$ are at $q=0$ or when $-q$ is a $N$-th root of unity for some
\[N\in \big\{H\cdot \beta'\colon 0<\beta'\leq \beta\big\}\,,\]
where $H$ is any ample divisor in $X$. It is conjectured in \cite[Conjecture 1.1]{schimpfBethe} that the poles are actually much more restricted.\footnote{The conjecture by Schimpf corrects an earlier conjecture in \cite{survey}, as explained in \cite[Remark 1.2.2]{schimpfBethe}.}
\end{remark}

\subsection{Wall-crossing and strategy of proof}\label{subsec: proofintro}

Our proof uses the same basic strategy as Toda's proof of Conjecture \ref{conj: rationality} in the Calabi--Yau case, but requires new ingredients: a wall-crossing formula for descendent integrals and a more involved combinatorial analysis. 

One rough heuristic behind Toda's proof is the following observation: taking the derived dual of a stable pair $I$ with Chern character $(-1,0,\beta, m)$ produces a new complex $I^\vee[2]$ with Chern character $(-1,0,\beta, -m)$. If by a miracle we knew that this involution sent stable pairs to stable pairs, then $Z_\beta^{\PT}(q|D)$ would actually be a Laurent polynomial satisfying the $q\leftrightarrow q^{-1}$ symmetry. This is typically not true, but Toda describes a path of stability conditions $\{\mu_s\}_{s\in \BR}$ that connects the moduli of stable pairs (when $s\rightarrow +\infty$) with the moduli of duals of stable pairs (when $s\rightarrow -\infty$). In the middle of this path, at $s=0$, we find the moduli stack of $\mu_0$-semistable objects $\FL_{\beta, m}^0$, which are indeed closed under taking the derived dual. The difference between $\mu_0$-semistable objects and stable pairs is controlled by semistable 1-dimensional sheaves. The invariants counting $\mu_0$-semistable objects will be called $\sL$ invariants.  

We will formulate the wall-crossing formula between $\PT$ invariants and $\sL$ invariants using a Joyce style (cf. \cite{Jo17, joyce}) vertex algebra $\VX$, which we review in Section \ref{sec: descendentsVA}. An element in $\VX$ is a functional on tautological classes; for example, the moduli of stable pairs $P_{\beta, m}$ defines an element $\PT_{\beta, m}\in \VX$ by
\begin{equation}\label{eq: PTfunctional}
   \PT_{\beta, m}\colon \quad \ch_{k_1}(\gamma_1)\ldots\ch_{k_n}(\gamma_n)\mapsto \int_{[P_{\beta, m}]^\vir}\ch_{k_1}(\gamma_1)\ldots\ch_{k_n}(\gamma_n)\in \BQ\,.
\end{equation}

Wall-crossing formulas are written using a ``Lie bracket'' $[-,-]$ on $\VX$ (see Section \ref{subsubsec: liebracket} for its definition). We will prove the following wall-crossing formula in Section \ref{sec: wallcrossing}:

\begin{thm}\label{thm: wcintro}
Let $X, \beta$ be as in the statement of Theorem \ref{thm: main}. We have an identity
    \[\PT_{\beta, m}=\sum_{\substack{\beta_0+\ldots+\beta_k=\beta\\
    m_0+\ldots+m_k=m\\
    0\leq\frac{m_1}{\beta_1\cdot H}\leq \ldots\leq \frac{m_k}{\beta_k\cdot H}}}\omord\Big(\frac{m_1}{\beta_1\cdot H},\ldots, \frac{m_k}{\beta_k\cdot H}\Big)\big[\mathsf{M}_{\beta_k, m_k},\big[\ldots,\big[\mathsf{M}_{\beta_1, m_1},\mathsf{L}_{\beta_0, m_0}\big]\ldots \big]\big]\,\]
in $\VX$, where $\sM_{\beta, m}\in \VXhat$ and $\sL_{\beta, m}=\sL_{\beta, m}^0\in \VX$ are ``generalized homological invariants'' counting $\mu_H$-semistable 1-dimensional sheaves and $\mu_0$-semistable complexes, respectively, and $\omega_{\textup{ord}}$ are the combinatorial coefficients in Definition \ref{def: omord}. Moreover, $\sM_{\beta, m}, \sL_{\beta, m}$ are well behaved with respect to tensoring by the line bundle $T_H(-)=H\otimes -$ and with respect to the derived dual $\delta(-)=(-)^\vee[2]$, i.e.
  \[(T_H)_\ast \sM_{\beta, m}=\sM_{\beta, m+H\cdot \beta}\, , \quad \delta_\ast \sM_{\beta, m}=\sM_{\beta, -m} \quad\textup{and}\quad \delta_\ast \sL_{\beta, m}=\sL_{\beta, -m}\, . \]
\end{thm}

The reason why we call them ``generalized homological invariants'' is that the moduli spaces $M_{\beta, m}$ or $L_{\beta, m}$ might include strictly semistable objects, in which case these classes cannot be simply defined as in \eqref{eq: PTfunctional}. While Joyce's machinery in \cite{joyce} gives a definition of $\sM_{\beta, m}$, the $\sL$ invariants are more difficult due to the absence of natural framing functors; see the discussion in the introduction of \cite{KM}, where the authors introduced a way to define ``generalized $K$-theoretic invariants''. 

Instead, our approach to Theorem \ref{thm: wcintro} is that we define the $\sL$ classes so that they satisfy the wall-crossing formula of Theorem \ref{thm: wcintro} and use the $K$-theoretic machinery from \cite{KM} to deduce the desired properties, namely the fact that they are well-behaved with respect to the derived dual and that, when every object of $\FL_{\beta, m}^s$ is stable, $\sL_{\beta, m}^s$ really corresponds to integration of tautological classes against the virtual fundamental class of the good moduli space $L_{\beta, m}^s$, cf. Theorem~\ref{thm: Linvariants}. 

Since algebraic $K$-theory has the difficulty that it does not produce odd cohomology classes, we actually use a version of our previous work, discussed in \cite[Appendix A]{KM}, that replaces algebraic $K$-theory by topological $K$-theory, as defined by \cite{B16}. For a smooth projective scheme $X$, Blanc's topological $K$-theory is isomorphic to $H^\ast(X)$ via the Chern character map, which allows us to deduce homological wall-crossing formulas in the generality needed.

Applying the results of \cite{KM} requires some technical conditions on the stacks of $\mu_s$-semistable objects, namely that they admit a proper good moduli space (Theorem \ref{thm: goodmoduli}), that they are quasi-smooth (Theorem \ref{thm: quasismooth}) and that they are the semistable locus of some pseudo $\Theta$-stratification (Theorem \ref{thm: Thetastratification}). This technical work is done in Section \ref{sec: technicalstacks}. We remark that the quasi-smoothness is the place where the positivity assumptions that we make in Theorem \ref{thm: main} are crucially used. 

\subsubsection{Combinatorial analysis}

Theorem \ref{thm: wcintro} is the only geometric input required for the proof of Theorem \ref{thm: main}, and the rest of the proof consists of unpacking the Lie brackets and combinatorics. If we fix $\beta$ in the $\PT/\sL$ wall-crossing formula, we will get finitely many possible partitions $\{\beta_i\}$ of $\beta$ contributing to the sum. Once this partition is fixed, we analyze its contribution to $Z^\PT_\beta(q|D)$, which is a sum over the cone defined by inequalities $0\leq\frac{m_1}{\beta_1\cdot H}\leq \ldots\leq \frac{m_k}{\beta_k\cdot H}$. 

Generating series obtained by summing over cones are the subject of Ehrhart theory \cite{ehrhart, stanley, BRcomputingcontinuous}, and indeed we end up deducing the rationality and symmetry of $Z^\PT_\beta(q|D)$ from general theorems about generating functions defined by sums over cones. The sum we are interested in is weighted by the coefficients $\omord$, so we extend Stanley's reciprocity theorem from Ehrhart theory by allowing a combination of quasi-polynomial weights and face-weights. If $C$ is a rational cone, $\omega$ is a face-weight (i.e. a constructible function on $C$, see Definition~\ref{def: faceweight}) and $f$ is a quasi-polynomial, we will show in Theorem \ref{thm: reciprocityII} that
\begin{equation}\label{eq: Zbothweightsintro}
    Z_{C, \hspace{0.02cm} \omega, f}(\bx)=\sum_{\bn\in C\cap \BZ^k} \omega(\bn)f(\bn)\bx^{\bn}\,.
\end{equation}
is the expansion of a rational function and that it satisfies a reciprocity theorem of the form 
  \[Z_{C, \hspace{0.02cm} \omega, f}(\bx^{-1})=Z_{C, \hspace{0.02cm} \omega^\vee\hspace{-0.05cm}, f^\vee\hspace{-0.05cm}}(\bx)\,.\]

The wall-crossing coefficients $\omord$ are a self-dual face-weight, and we use this to deduce Corollary \ref{cor: combinatorialsymm}, which is the combinatorial input used in the proof of Theorem \ref{thm: main}. Section \ref{sec: ehrhart} is completely independent from the rest of the paper and can be read by someone who does not know what a sheaf is.

\subsection{Primary insertions}\label{subsec: primaryintro}
Stable pair invariants that only use $\ch_2(\gamma)$ as insertions (i.e. \eqref{eq: D} with $k_i=2$ for all $i$) are called primary invariants. These are more geometric: the classes $\ch_2(\gamma)$ correspond to imposing incidence conditions on the support of the stable pair. 

The famous $\GW/\PT$ correspondence \cite{mnop1, mnop2} has been established by Pardon \cite{pardon} for primary insertions on Fano 3-folds. As in the Calabi--Yau case, primary insertions on the GW side have long been expected to be governed by Gopakumar--Vafa (GV) or BPS invariants \cite{3questions}, and this has been proven to be true \cite{zinger, IPgvinvariants, DW, DIW}. The $\GW$/$\GV$/$\PT$ correspondence implies a very strong rationality result for the generating series $Z^\PT_\beta(q|D)$ when $D$ is primary; for instance, the only pole (other than $q=0$) is $q=-1$, while for general insertions we often have poles at roots of unity. 

Specializing our general wall-crossing formula to primary insertions yields a substantial simplification, making the wall-crossing formula resemble a lot its Calabi--Yau counterpart. We will prove in Theorem \ref{thm: primarywc} the equality
\begin{equation}\label{eq: wcprimary}Z^{\PT}_{\textup{prim}}(Q, q, \bt)=\exp\left(\frac{1}{q^{1/2}+q^{-1/2}}Z^{\mathsf{M}}_{\textup{prim}}(Q, \bt)\right)Z^{\sL}_{\textup{prim}}(Q, q, \bt)\,\end{equation}
where the terms are defined in \eqref{eq: partitionfunctions}. Roughly speaking, $Z^{\PT}_{\textup{prim}}(Q, q, \bt)$ is the generating series of all primary $\PT$ invariants, with the variable $q$ being used to keep track of $m=\ch_3$, $Q$ keeping track of $\beta=\ch_2$ and $\bt$ keeping track of the primary insertions. The series $Z^{\mathsf{L}}_{\textup{prim}}(Q, q, \bt)$ is defined in a similar way but replacing $\PT$ invariants with $\sL$ invariants. The series $Z^{\mathsf{M}}_{\textup{prim}}(Q, \bt)$ contains all the information concerning primary invariants on moduli of 1-dimensional sheaves $M_{\beta, m}$, but only for curve classes $\beta$ with $d_\beta=1$! In particular, when there are no such classes (e.g. if $X=\BP^3$) we simply have $Z^{\mathsf{M}}_{\textup{prim}}=0$, and thus primary $\PT$ invariants agree with primary $\sL$ invariants.

Formula \eqref{eq: wcprimary} implies a very strong rationality, and moreover identifies some genus 0 GV invariants of Fano 3-folds. It is better stated in terms of connected $\PT$ invariants, which are defined formally from usual $\PT$ invariants by essentially taking a logarithm of $Z^{\PT}_{\textup{prim}}(Q, q, \bt)$, see Section \ref{subsec: primaryrationality} for the precise definition.

\begin{thm}\label{thm: primarystrongrationality}
Let $X, \beta$ be as in Theorem \ref{thm: main} and let $D$ be a primary insertion. If $d_\beta>1$ then the connected $\PT$ generating series $Z_{\beta}^{\PT, \circ}(q|D)$ is a Laurent polynomial in $q$. If $d_\beta=1$ then
\[Z_{\beta}^{\PT, \circ}(q|D)=\frac{1}{q^{1/2}+q^{-1/2}}\int_{\sM_\beta} D+\textup{ Laurent polynomial in }q\,.\]
\end{thm}

The theorem says that the only source of poles in primary $\PT$ series are curve classes $\beta$ with $d_\beta=1$, which is consistent with the $\PT$/$\GV$ correspondence \cite[(3.25)]{PT}. Moreover, it identifies the genus 0 GV invariant for such curve classes as an integral over a moduli of 1-dimensional sheaves
\[n_{g=0, \beta}(\gamma)=\int_{\sM_{\beta}}\ch_2(\gamma)\in \BZ\quad \textup{ if }d_\beta=1\,.\]
This is analogous to a conjecture of Katz in the Calabi--Yau case. Toda showed in \cite[Theorem 6.4]{todacurvecounting}, using the Calabi--Yau version of the $\PT$/$\sL$ wall-crossing formula, that Katz' conjecture is equivalent to the $\chi$-independence (or $m$-independence, in our notation) of some invariants defined out of $M_{\beta, m}$. For us, this $\chi$-independence is straightforward (cf. Lemma~\ref{lem: easychiindependence}) since we are only concerned with curve classes such that $d_\beta=1$ (which in particular implies that $\beta$ is irreducible, so stability in curve class $\beta$ is equivalent to purity).

\subsection{An example: lines in $\BP^3$}\label{subsec: introexample}
We will illustrate the wall-crossing formula of Theorem \ref{thm: wcintro} and the proof of Theorem \ref{thm: main} in the case that $X=\BP^3$ and $\beta$ is the class of a line. The virtual dimension in this example is 4 and $P_{1, m}$ is empty for $m< -1$. When $m\geq -1$ the moduli space $P_{1,m}$ is a $\BP^{m+1}$ bundle over the Grassmannian $\textup{Gr}(2, \BC^4)$. Stable pair invariants can be explicitly computed by either using this geometric description (and identifying the virtual fundamental class, as in \cite[Section 6.2]{moreira}) or by localization. The generating series $Z_{\beta=1}$ are recorded in \cite[Appendix A]{MOOP}; as an example we will take $D=\ch_7({\bf 1})$, in which case 
    \begin{equation}\label{eq: examplegeneratingseries}Z_{\beta=1}(q|\ch_7({\bf 1} ))=\frac{q^{-1} (q - 1) (2 + 3 q - 28 q^2 + 3 q^3 + 2 q^4)}{18 (1 + q)^3}\,.
    \end{equation}

The wall-crossing formula in this case is
    \[\sum_{m\geq -1} q^m\PT_{1, m}=q^{-1}\sL_{1, -1}+\sL_{1, 0}+q\sL_{1, 1}+\sum_{m>0}q^m[\sM_{1,m},\sL_{0,0}]+\frac{1}{2}[\sM_{1,0},\sL_{0,0}]\,.\]
    The class $\sL_{0,0}$ in the formula is very simple: it corresponds to the moduli space with a single point $\{\CO_X[1]\}$. We only have $\sL_{1, m}$ for $m=-1,0,1$ since $L_{\beta, m}=P_{\beta, m}$ is empty for $m<-1$, and by symmetry of the $L$ invariants $L_{\beta, m}$ is empty also for $m>1$.

For any fixed descendent $D\in \BD^X$ the contribution $(-1)^m \int_{[\sM_{1,m},\sL_{0,0}]} D$ is polynomial in $m$ (for higher degree $\beta$ it is actually a quasi-polynomial). Indeed, from the definition of $[-,-]$ one finds that 
    \[\int_{[\sM_{1,m},\sL_{0,0}]} \ch_7({\bf 1} )=\frac{(-1)^m}{6}\int_{\sM_{1,m}}\ch_4({\bf 1})-\ch_3({\bf 1})\ch_3(H)\,. \]
Note that tensoring by line bundles defines isomorphisms $-\otimes \CO_{\BP^3}(m)\colon M_{1,0}\xrightarrow{\sim} M_{1, m}$; indeed, $M_{1,m}$ are all isomorphic to the space of lines in $\BP^3$, i.e. the Grassmannian $\textup{Gr}(2, \BC^4)$, with a non-trivial obstruction bundle that makes its virtual dimension 1.  These ``twist by a line bundle'' isomorphisms are the source of quasi-polynomiality, cf. Proposition \ref{prop: Mquasipolynomial}. In our example
\begin{align*}
\int_{\sM_{1,m}}&\ch_4({\bf 1})-\ch_3({\bf 1})\ch_3(H)\\
&=\int_{\sM_{1,0}}\ch_4({\bf 1})+m\ch_3(H)+\frac{m^2}{2}\ch_2(H^2)-\big(\ch_3({\bf 1})+m\ch_2(H)\big)\big(\ch_3(H)+m\ch_2(H^2)\big)
\end{align*}
is evidently a polynomial of degree $\leq 2$.

    Since the $\sL$ invariants only contribute for $m\leq 1$, we can recover this polynomial by expanding \eqref{eq: examplegeneratingseries}:
    \[\int_{[\sM_{1,m},\sL_{0,0}]} \ch_7({\bf 1} )=\frac{5}{9}(-1)^m(1-3m^2)\,.\]
    Observe that there is no linear term and that the polynomial $1-3m^2$ is even, which is a consequence of the duality symmetry $\delta_\ast \sM_{1, m}=\sM_{1,-m}$. The $\sL$ invariants can also be recovered from \eqref{eq: examplegeneratingseries}:
    \[\int_{\sL_{1,1}}\ch_7({\bf 1} )=\frac{1}{9}=-\int_{\sL_{1,-1}}\ch_7({\bf 1} )\quad \textup{ and }\quad  \int_{\sL_{1,0}}\ch_7({\bf 1} )=0\,.\]
    The symmetry between $\sL_{1,1}$ and $\sL_{1,-1}$, and the vanishing for $\sL_{1,0}$ are explained again by the symmetry $\delta_\ast \sL_{1,m}=\sL_{1, -m}$. 

\begin{remark}
There is a natural projection $P_{1,m}\to M_{1,m}$ that forgets the section. When $m>1$ this map can be seen to be a virtual projective bundle, and \cite[Proposition 6.2]{blm} shows that $\PT_{1,m}=[\sM_{1,m}, \sL_{0,0}]$, which is precisely the wall-crossing formula when no $\sL_{1,m}$ invariants contribute. This phenomenon can also be seen in the calculation of stable pairs supported on lines in the cubic 3-fold in \cite[Sections 5, 6]{moreira}, where the $m=0$ case has to be treated separately and $m>0$ becomes uniform since the moduli of stable pairs becomes a virtual projective bundle over the moduli of 1-dimensional sheaves. The term $\Lambda^2 \mathcal S^\vee$ in \cite[Proposition 24]{moreira}, which is the obstruction bundle on $F(X)$, is explained by this virtual projective bundle structure.
\end{remark}

\subsection{Comparison with \cite{AJ}}\label{subsec: comparison}
Independently and simultaneously, R. Anderson and D. Joyce have shown that the rationality part of Conjecture \ref{conj: rationality} holds for any curve class $\beta$ as in the statement of Theorem \ref{thm: main} (which they call a ``superpositivite'' curve class), but without the constraint on Hodge numbers. They also work equivariantly. On the other hand, their methods do not prove the functional equation. 

While their proof also uses wall-crossing, it is a different type of wall-crossing: they treat stable pairs as actual pairs, and not as an object in the derived category, and use the results from an updated (to appear) version of \cite{joyce}. Their method is similar in spirit to the rank reduction procedure that appears for example in work of Thaddeus \cite{thaddeus} for curves or of Mochizuki \cite{mochizuki} for surfaces.  

The pair obstruction theory matches the obstruction theory of complexes when $m>C_\beta$ for some constant $C_\beta$, and controlling $\PT_{\beta, m}$ for large $m$ is enough to prove rationality, but not the functional equation. In our approach, we truly wall-cross in the derived category and make use of the foundational developments in \cite{KM}, which allows us to control all $\PT_{\beta, m}$ with $\sL$ invariants.

\subsection{Acknowledgments} We would like to thank Reginald Anderson, Daniel Halpern--Leistner, Dominic Joyce, Woonam Lim, Davesh Maulik, Rahul Pandharipande and Weite Pi for conversations related to stable pairs and wall-crossing. 

We especially thank Reginald Anderson and Dominc Joyce for cordial communications regarding their independent work \cite{AJ}.

\section{Reciprocity theorems for weighted sums over cones}\label{sec: ehrhart}

We prove in this section some elementary combinatorial results regarding generating series defined by sums over integral points on cones, extending slightly well known results in Ehrhart theory -- a good introductory reference for the subject is for example \cite{BRcomputingcontinuous}. We recall some standard terminology here: a (rational) cone $C\subseteq \BQ^k$ is the intersection of finitely many (open or closed) half spaces $\langle y, a\rangle\geq 0$ or $\langle y, a\rangle> 0$ for some $a\in \BQ^k\setminus \{0\}$. A face of a closed cone $C$ is the intersection of $C$ with a supporting hyperplane; we denote by $\dim(F)$ the dimension of a face $F$. We denote by $F^\circ=F\setminus \bigcup_{F'\subsetneq F} F'$ the relative interior of $F$; note that $C$ is the disjoint union of $F^\circ$ over all its faces $F$. 

\begin{definition}\label{def: faceweight}
    Let $C$ be a closed cone. A face-weight $\omega$ on $C$ is an assignment of a rational number $\omega(F)\in \BC$ to every face $F$ of $C$. 

    Given a face-weight $\omega$, we define its dual weight $\omega^\vee$ by
    \[\omega^\vee(F)=\sum_{G\supseteq F}(-1)^{\dim G}\omega(G)\,.\]
    We say that $\omega$ is self-dual if $\omega^\vee=(-1)^{\dim C}\omega$. 
    \end{definition}

A face-weight can also be regarded as a constructible function on $C$, by letting $\omega(\bn)=\omega(F)$ where $F$ is the face whose relative interior contains $\bn\in C$. 

\begin{example}\label{ex: standardweight}
If $\omega(F)=1$ for all $F$ then 
\[\omega^\vee(F)=\begin{cases}
    (-1)^{\dim C}& \textup{ if }F=C\\
    0 & \textup{otherwise}
\end{cases}\,.\]
We call this $\omega$ the standard weight.
\end{example}

\begin{example}\label{ex: faceweighteasy}
Let $C$ be a simplicial cone of dimension $k$. The poset of faces of $C$ is opposite to the poset of subsets of $[k]\coloneqq \{1, 2, \ldots, k\}$, and we denote by $F_I$ the face of $C$ corresponding to a subset $I\subseteq [k]$; we have $\dim F_I=k-|I|$. For the sake of concreteness we may assume $C$ is given by the intersection of the half-spaces $y_i\geq 0$ and $F_I$ is the intersection with the hyperplanes $y_i=0$ for $i\in I$. Then the following two face-weights are self-dual:
\[\omega_1(F_I)=\frac{1}{2^{|I|}}\textup{ and }\omega_2(F_I)=\frac{1}{|I|+1}\,.\]
This follows from the identities
\[\sum_{j=0}^\ell (-1)^j\frac{1}{2^j}\binom{\ell}{j}=\frac{1}{2^\ell}\textup{ and }\sum_{j=0}^\ell (-1)^j\frac{1}{j+1}\binom{\ell}{j}=\frac{1}{\ell+1}\,,\]
which are both consequences of the binomial theorem.
\end{example}

For our next example we introduce some notation:

\begin{definition}\label{def: omord}
Consider real numbers $0\leq \mu_1\leq \mu_2\leq\ldots\leq \mu_k$. For convenience, set $\mu_0=0$. We define 
\[\omord(\mu_1, \ldots, \mu_k)=\prod_{j=1}^l \frac{1}{(c_j-c_{j-1})!}\]
where $-1=c_{0}<c_1<\ldots<c_l=k$ are such that the following holds: for $i=0, \ldots, k-1$ we have $\mu_{i}<\mu_{i+1}$ if and only if $i=c_j$ for some $j=1, \ldots, l-1$.
\end{definition}

In words, $i=c_j$ are the places where $\mu_i$ jumps and $\omord(\mu_1, \ldots, \mu_k)$ is the product of inverse of the factorials of lengths of chains of equalities in $0\leq \mu_1\leq \mu_2\leq\ldots\leq \mu_k$. For example, if
\[0=\mu_1=\mu_2<\mu_3<\mu_4=\mu_5\]
then $\omord(\mu_1, \ldots, \mu_5)=\frac{1}{3!1!2!}=\frac{1}{12}$.

\begin{example}\label{ex: faceweightwc}
We consider the cone defined by the inequalities $0\leq \mu_1\leq \ldots\leq \mu_k$, which is also a simplicial cone of dimension $k$. Then $\omord(\mu_1, \ldots, \mu_k)$ only depends on the face in which $\boldsymbol{\mu}=(\mu_1, \ldots, \mu_k)$ lies. The faces of the cone are indexed again by subsets $I\subseteq [k]$; $F_I$ is the face obtained by intersecting $C$ with the hyperplanes $\mu_i=\mu_{i-1}$ for $i\in I$, where we set $\mu_0=0$. We write $I$ as a union of consecutive blocks
\[I=\bigcup_{i=1}^l [a_i, b_i]\]
for some $a_i, b_i$ with $a_i<b_i<a_{i+1}-1$. Then
\[\omord(\boldsymbol{\mu})=\omord(F_I)=\prod_{i=1}^l \frac{1}{(b_i-a_i+2)!}\,\]
if $\boldsymbol{\mu}\in F_I^\circ$. In the example above $I=[1,2]\cup [5,5]$.
\end{example}

\begin{proposition}\label{prop: weightselfdual}
    The face-weight of Example \ref{ex: faceweightwc} is self-dual.
\end{proposition}
\begin{proof}
Let $\Sigma_{k+1}$ be the set of permutations of $[k+1]$ and, for each $i\in [k]$, we let
\[A_i=\{\sigma\in \Sigma_{k+1}\colon \sigma(i)<\sigma(i+1)\}\,,\quad B_i=\{\sigma\in \Sigma_{k+1}\colon \sigma(i)>\sigma(i+1)\}\,.\]
For $I\subseteq [k]$ we let $A_I$ be the intersection of $A_i$ for $i\in I$, and similarly define $B_I$. The face-weight of Example \ref{ex: faceweightwc} admits a natural interpretation as
\[\omord(F_I)=\frac{|A_I|}{(k+1)!}=\frac{|B_I|}{(k+1)!}\,.\]
Since $B_i=\Sigma_{k+1}\setminus A_i$ we have by the principle of inclusion-exclusion that
\[\omord(F_I)=\frac{1}{(k+1)!}|B_I|=\frac{1}{(k+1)!}\big|\Sigma_{k+1}\setminus \bigcup_{i\in I} A_i\big|=\frac{1}{(k+1)!}\sum_{J\subseteq I}(-1)^{|J|}|A_J|=(-1)^k \omord^\vee(F_I)\,.\qedhere\]
\end{proof}

\begin{remark}
The proof of Proposition \ref{prop: weightselfdual} illustrates a general way to construct (self-dual) face-weights. Given a sequence of probability events $\{A_i\}_{1\leq i\leq k}$ we may define a face-weight on simplicial cones by $\omega(F_I)=\Pr(\bigcap_{i\in I} A_i)$. Then $(-1)^k \omega^\vee$ is constructed in the same way from the complementary events $\{A_i^c\}_{1\leq i\leq k}$. If there is a probability preserving involution of the underlying probability space which sends $A_i$ to $A_i^c$, then $\omega$ is self-dual. Example \ref{ex: faceweighteasy} can be interpreted in this fashion as well.
\end{remark}

\subsection{Ehrhart reciprocity with face-weights}\label{subsec: ehrhartfaceweight} Let $C\subseteq \BQ^k$ be a cone and $\omega$ a face-weight on $C$. We consider the formal variables ${\bf x}=(x_1, \ldots, x_k)$. Given $\bn=(n_1, \ldots, n_k)\in \BZ^k$ we denote by $\bx^\bn$ the monomial
\[\bx^\bn=x_1^{n_1}\ldots x_k^{n_k}\,.\]
Recall that a face weight $\omega$ defines a constructible function $C\to \BQ$ by setting $\omega(\bn)=\omega(F)$ for $\bn\in F^\circ$. We define a generating series of integral points in $C$, weighted by $\omega$, as follows:
\begin{equation}\label{eq: Zfaceweight}
    Z_{C, \hspace{0.02cm} \omega}(\bx)=\sum_{\bn\in C\cap \BZ^k} \omega(\bn)\bx^{\bn}\in \BC \llbracket \bx^{\pm 1}\rrbracket\,.
\end{equation}
Here, $\BC \llbracket \bx^{\pm 1}\rrbracket=\BC\llbracket x_1^{\pm 1}, \ldots, x_k^{\pm k}\rrbracket$. When $\omega=1$ is the standard weight we simply write $Z_C$. Such generating series are known to be expansions of rational functions in $\bx$; we recall that a formal power series $Z(\bx)\in \BC \llbracket \bx^{\pm 1}\rrbracket$ is an expansion of a rational function $f(\bx)/g(\bx)$ if $g(\bx)Z(\bx)=f(\bx)$. 

Since $C$ is the disjoint union of the relative interiors of its faces we can write
\[Z_{C, \hspace{0.02cm} \omega}(\bx)=\sum_{F\subseteq C} \omega(F)Z_{F^\circ}(\bx)\]
where the sum is over all faces of $C$. In particular $Z_{C, \hspace{0.02cm} \omega}(\bx)$ is also the expansion of a rational function.

\begin{theorem}\label{thm: reciprocityI} We have an equality of rational functions
\[Z_{C, \hspace{0.02cm} \omega}(\bx^{-1})=Z_{C, \hspace{0.02cm} \omega^\vee}(\bx)\,.\]
\end{theorem}
\begin{proof}
    We deduce this result from Stanley's reciprocity, which is the particular case where $\omega$ is the standard weight (cf. \cite[Theorem 4.3]{BRcomputingcontinuous} and Example \ref{ex: standardweight}). We have:
    \begin{align*}
        Z_{C, \hspace{0.02cm} \omega}(\bx^{-1})&=\sum_{F\subseteq C} \omega(F)Z_{F^\circ}(\bx^{-1})=\sum_{F\subseteq C} \omega(F)(-1)^{\dim F}Z_{F}(\bx)\\
        &=\sum_{F\subseteq C}\sum_{G\subseteq F} \omega(F)(-1)^{\dim F}Z_{G^\circ}(\bx)=\sum_{G\subseteq C}\omega^\vee(G)Z_{G^\circ}(\bx)=Z_{C, \hspace{0.02cm} \omega^\vee}(\bx)\,.\qedhere
    \end{align*}
\end{proof}

\subsection{Quasi-polynomial coefficients}\label{subsec: quasipolynomialehrhart}
We extend now Theorem \ref{thm: reciprocityI} by allowing not only the combinatorial weights $\omega$ but also a quasi-polynomial dependence on $\bn$. Recall that $f\colon \BZ^k\to \BC$ is a quasi-polynomial if for some $N\in \BZ_{\geq 1}$ and every ${\bf a}\in \BZ^k$ the function $\bn\mapsto f(N\bn+{\bf a})$ is a polynomial function in $\bn$; the smallest such integer $N$ is called the period of $f$. Given such a quasi-polynomial $f$ we write
\begin{equation}\label{eq: Zbothweights}
    Z_{C, \hspace{0.02cm} \omega, f}(\bx)=\sum_{\bn\in C\cap \BZ^k} \omega(\bn)f(\bn)\bx^{\bn}\in \BC \llbracket \bx^{\pm 1}\rrbracket\,.
\end{equation}
We define the quasi-polynomial $f^\vee$ by
\[f^\vee(n_1, \ldots, n_k)=f(-n_1, \ldots, -n_k)\,.\]

\begin{theorem}\label{thm: reciprocityII}
    Let $C$ be a cone, $\omega$ a face-weight and $f$ a quasi-polynomial. Then $Z_{C, \hspace{0.02cm} \omega, f}(\bx)$ is the expansion of a rational function. Moreover, we have the equality of rational functions
    \[Z_{C,\hspace{0.02cm} \omega, f}(\bx^{-1})=Z_{C, \hspace{0.02cm} \omega^\vee\hspace{-0.05cm}, f^\vee\hspace{-0.05cm}}(\bx)\,.\]
\end{theorem}
\begin{proof}
Any quasi-polynomial of period $N$ can be expressed as a linear combination of quasi-polynomials of the form
\[f(\bn)=\prod_{i=1}^k n_i^{a_i} \zeta^{b_i n_i}\]
for some non-negative integer $a_i, b_i$, where $\zeta$ is a primitive $N$-th root of unity. So it is enough to prove the result for such $f$. Note that $f^\vee$ is of the same shape, up to a sign:
\[f^\vee (\bn)=(-1)^{\sum_{i=1}^k a_i}\prod_{i=1}^k n_i^{a_i}\zeta^{-b_i n_i}\,.\]

Let $\partial_{i}$ be the differential operator $x_i\frac{\partial }{\partial x_i}$. Then, for $f$ of the form above, we have
\[Z_{C,\hspace{0.02cm}  \omega, f}(\bx)=\partial_{1}^{a_1}\ldots \partial_{k}^{a_k}Z_{C, \hspace{0.02cm} \omega}(\bx)|_{x_i\mapsto \zeta^{b_i}x_i}\,.\]
Here, the notation on the right means that we first differentiate $Z_{C, \hspace{0.02cm} \omega}(\bx)$ and, after differentiation, we replace the variable $x_i$ by $\zeta^{b_i}x_i$.The operators $\partial_{i}$ have the property that they preserve rational functions and for any $Z$ we have
\[\partial_{i} Z(\bx^{-1})=-\partial_{i}Z(\bx)|_{\bx\mapsto \bx^{-1}}\,.\]
Hence
\begin{align*}
Z_{C, \hspace{0.02cm} \omega, f}(\bx^{-1})&=\partial_{1}^{a_1}\ldots \partial_{k}^{a_k}Z_{C, \hspace{0.02cm} \omega}(\bx)|_{x_i\mapsto \zeta^{b_i}x_i^{-1}}=(-1)^{\sum_{i=1}^k a_i}\partial_{1}^{a_1}\ldots \partial_{k}^{a_k}Z_{C, \hspace{0.02cm} \omega}(\bx^{-1})|_{x_i\mapsto \zeta^{-b_i}x_i}\\
&=Z_{C, \hspace{0.02cm} \omega^\vee\hspace{-0.05cm}, f^\vee\hspace{-0.05cm}}(\bx)\,.\qedhere
\end{align*}
\end{proof}

\subsection{Corollaries}\label{subsec: ehrhartcorollaries} We discuss now two very concrete particular cases of Theorem \ref{thm: reciprocityII}. In particular, Corollary \ref{cor: combinatorialsymm} is what we will use in the proof of the rationality of $\PT$ invariants.

\begin{corollary}\label{cor: combinatorialsymmeasy}
Let $f_1, \ldots, f_k\colon \BZ\to \BC$ be quasi-polynomials. Then the generating series
\[Z(q)=\sum_{n_1, \ldots, n_k\geq 0}\,\frac{1}{\#\{i\colon n_i=0\}+1}f_1(n_1)\ldots f_k(n_k) q^{n_1+\ldots+n_k}\]
is the expansion of a rational function.

Moreover, if all $f_i$ are either even or odd, then this rational function satisfies the symmetry
\[Z(q^{-1})=(-1)^a Z(q)\]
where $a$ is the number of even ones.
\end{corollary}
\begin{proof}
    The generating series above is $Z_{C, \hspace{0.02cm} \omega_2, f}$ where $f(\bn)=\prod_{i=1}^k f_i(n_i)$ and $C, \omega_2$ are as in Example \ref{ex: faceweighteasy}. We have $\omega_2^\vee=(-1)^k \omega_2$ and $f^\vee=(-1)^{k-a}f$, so we are done.\qedhere
\end{proof}

For our last corollary recall Definition \ref{def: omord} of $\omord$.

\begin{corollary}\label{cor: combinatorialsymm}
Let $r_1, \ldots, r_k$ be positive integers and let $f_1, \ldots, f_k\colon \BZ\to \BC$ be quasi-polynomials. Then the generating series
\[Z(q)=\sum_{0\leq \frac{n_1}{r_1}\leq \ldots\leq \frac{n_k}{r_k}}\omord\Big(\frac{n_1}{r_1}, \ldots, \frac{n_k}{r_k}\Big)f_1(n_1)\ldots f_k(n_k) q^{n_1+\ldots+n_k}\]
is the expansion of a rational function.

Moreover, if all $f_i$ are either even or odd, then this rational function satisfies the symmetry
\[Z(q^{-1})=(-1)^a Z(q)\]
where $a$ is the number of even ones.
\end{corollary}
\begin{proof}
    Same as the proof of Corollary \ref{cor: combinatorialsymmeasy} using the rational cone defined by the inequalities $0\leq \frac{n_1}{r_1}\leq \ldots\leq \frac{n_k}{r_k}$ and Proposition \ref{prop: weightselfdual}.\qedhere
\end{proof}

\begin{remark}\label{rmk: combinatorialpoles}
The simplicial cone considered in Corollary \ref{cor: combinatorialsymm} is generated over $\BQ$ by the vectors $(0,\ldots, 0, r_{i}, \ldots, r_k)$ for $1\leq i\leq k$, which can be used to control the poles appearing using \cite[Theorem 3.5]{BRcomputingcontinuous}. Suppose further that the period of $f_i$ divides $r_i$. Then the poles of $Z(q)$ are all $N$-th roots of unity where
\[N\in \left\{\sum_{j=i}^k r_j\colon 1\leq i\leq k\right\}\,.\]
\end{remark}

\section{Descendents and the vertex algebra}
\label{sec: descendentsVA}
In this section we review the wall-crossing setup from \cite{joyce}, following a fairly concrete approach in terms of the descendent algebra as in \cite{blm}. 

\subsection{The descendent algebra and the realization homomorphism}
\label{subsec: descendents}

Given a smooth projective $X$, we use the descendent algebra $\BD^X$ as a tool to keep track of tautological classes. We mostly follow the notation from \cite[Section 2]{blm} (except we do not use the Hodge shift) and refer the reader there for more details. For this section we let $d=\dim(X)$, which will be~3 throughout the rest of the paper.

\begin{definition}\label{def: descendentalgebra}
    The descendent algebra $\BD^X$ is the  (super)commutative $\BQ$-algebra generated by symbols $\ch_k(\gamma)$, for $k\geq 0$ and $\gamma\in H^\ast(X)$, modulo linearity relations
    \[\ch_k(\lambda_1\gamma_1+\lambda_2\gamma_2)=\lambda_1\ch_k(\gamma_1)+\lambda_2\ch_k(\gamma_2)\,.\]
\end{definition}

If $M$ is a moduli space/stack of perfect complexes (such as the moduli of stable pairs) on $X$ with a universal complex $\CF$ on $M\times X$ there exists a realization map
\[\xi_\CF\colon \BD^X\to H^\ast(X)\]
which sends
\[\ch_k(\gamma)\mapsto p_\ast\big(\ch_k(\CF)q^\ast \gamma\big)\in H^{2k+\deg(\gamma)-2d}(M)\,.\]
Above, $p, q$ are the projections of $M\times X$ onto $M$ and $X$, respectively. We denote by $K(X)$ the numerical Grothendieck group of $X$, i.e. the image 
\[K(X)=\textup{im}\big(K_0(D^b(X)\xrightarrow{\ch}H^\ast(X)\big)\,.\]
Given $\alpha\in K(X)$ let $\BD^X_\alpha$ be the quotient of $\BD^X$ by the relations 
\[\ch_k(\gamma)=\begin{cases}
    0 & \textup{if }2k+\deg(\gamma)<2d\\
    \int_{X}\alpha\cdot \gamma  &\textup{if }2k+\deg(\gamma)=2d
\end{cases}\]
If the objects of $M$ have Chern character $\alpha$ then the realization map factors through $\BD^X_\alpha\to H^\ast(M)$. We will refer to $\alpha\in K(X)$ as a topological type of objects of $M$.

\subsubsection{Weight 0 descendents and normalization}
\label{subsubsec: weight0}
Moduli spaces of sheaves often do not have a canonical universal sheaf. If $\CF$ is a universal sheaf, then $\CF\otimes p^\ast L$ is also a universal sheaf for any line bundle $L$ on $M$. Weight 0 descendents are the descendents which are not affected by this ambiguity. Formally, we define the subalgebra of weight 0 descendents by
\[\BD^X_\inv=\ker(\bR_{-1}\colon \BD^X\to \BD^X)\quad\textup{ and }\quad \BD^X_{\inv, \alpha}=\ker(\bR_{-1}\colon \BD^X_\alpha\to \BD^X_\alpha)\]
where $\bR_{-1}$ is a derivation defined on generators by 
\begin{equation}
    \label{eq: R-1}\bR_{-1}(\ch_k(\gamma))=\ch_{k-1}(\gamma)\,.
\end{equation}
Any moduli space of semistable sheaves or complexes of topological type $\alpha$ that does not contain strictly semistable objects defines a canonical realization map
\[\xi\colon \BD^X_{\inv,\alpha}\to H^\ast(M)\,.\]
See \cite[Section 2.4]{blm} for details. Extending this realization map to $\BD^X_\alpha$ corresponds, roughly speaking, to choosing a universal sheaf $\CF$. It is discussed in \cite[Section 2.5]{blm} and \cite[Section 1.2.1]{klmp} how to choose a lift by normalizing a certain class. If $\delta\in H^{2j}(X)$ is such that $\int_X \alpha\cdot \delta\neq 0$ then there exists a unique extension of $\xi$ to $\BD^X_\alpha$ which sends $\ch_{d-j+1}(\delta)$ to~0; indeed, this extension is defined as the composition $\xi\circ \eta_\delta$, where $\eta_\delta$ is as in \cite[Remark 2.14]{blm}. We call this the $\delta$-normalized extension.

\begin{example}\label{ex: ptnormalizedPT}
    Let $\CI=[\CO_{P_{\beta, m}\times X}\to \CF]$ be the universal stable pair on $P_{\beta, m}\times X$. Since $\CF$ is a sheaf supported in codimension 2, the realization map $\xi_\CI$ sends
    \[\ch_0(\pt)\mapsto -1\quad \textup{ and }\quad \ch_1(\pt)\mapsto 0\,.\]
    Therefore, $\xi_\CI$ is the $\pt$-normalized extension of $\xi$.
\end{example}

\subsection{Wall-crossing vertex algebra}
\label{subsec: wcVA}
Let $\FM_X$ be the moduli stack of perfect complexes on~$X$ defined in \cite{TV07}. Joyce defines in \cite{joyce} a vertex algebra structure on $H_\ast(\FM_X)$ which is used to write down wall-crossing formulas. For the sake of concreteness, we will mostly use the more explicit definition using the descendent algebra, as in \cite[Section 3.2]{parabolic}. As explained in loc. cit., there exists a natural vertex algebra structure on
\[\VX\coloneqq \bigoplus_{\alpha\in K(X)}(\BD_\alpha^X)^\vee\,,\]
where $(\BD_\alpha^X)^\vee$ stands for the graded dual of $\BD_\alpha^X$, where the grading is the cohomological grading. We write $\bV_\alpha$ for the summand $(\BD_\alpha^X)^\vee$. Dualizing the realization map $\BD^X_\alpha\to H^\ast(\FM_\alpha)$, where $\FM_\alpha\subseteq \FM_X$ is the union of connected components of complexes with Chern character $\alpha$, produces a map
\[H_\ast(\FM_X)\to \VX\,,\]
which is an homomorphism of vertex algebras. For some $X$ (e.g. $X$ toric), this homomorphism is actually an isomorphism \cite[Theorem 1.1]{Gr19}.

\subsubsection{Classes of moduli spaces}
\label{subsubsec: classesmoduliVA}
Let $M$ be a moduli space parametrizing semistable perfect complexes in $X$ with topological type $\alpha$. Assume $M$ has no strictly semistable objects, that $M$ is proper and that $M$ has a virtual fundamental class $[M]^\vir$. If $\CF$ is a universal complex on $M\times X$, then the functional 
\[\BD_\alpha \ni D\mapsto \int_{[M]^\vir} \xi_\CF(D)\]
defines an element of $\bV_\alpha\subseteq \bV_X$.

If we are not given a universal complex $\CF$, we have only a functional
\[\BD_{\inv,\alpha} \ni D\mapsto \int_{[M]^\vir} \xi(D)\,.\]
So $M$ determines an element of
\[\VXhat\coloneqq \bigoplus_{\alpha\in K(X)}(\BD_{\inv,\alpha}^X)^\vee\,.\]
Equivalently, $\VXhat$ is the quotient of $\VX$ by the image of the dual operator $\bR_{-1}^\vee$. Joyce's wall-crossing theory provides a way to associate an element of $\VXhat$ to moduli spaces of sheaves even when they contain strictly semistable objects, and we call such elements \textit{generalized homological invariants}.

\subsubsection{Lie bracket}\label{subsubsec: liebracket}
Wall-crossing formulas are expressed using the $0$-th product in $\VX$, which we will denote by
\[[-,-]\coloneqq \VX\otimes \VX\to \VX\,.\]
The $0$-th product satisfies the Jacobi identity, but it is not a Lie bracket on the nose since it is not anti-symmetric. We explain now the very concrete definition of $[-,-]$:

For $s\in \BZ$ define 
\[\Delta_s\colon \BD^X\to \BD^X\otimes \BD^X\]
by the formula
\begin{equation}
    \Delta_s=\sum_{j\geq 0}c_{s+j+1}(\Theta)\circ \left(\frac{\bR_{-1}^j}{j!}\otimes \id\right)\circ \Sigma^\ast\end{equation}
where
\begin{enumerate}
    \item $\Sigma^\ast\colon \BD^X\to \BD^X \otimes \BD^X$ is the algebra homomorphism defined on generators by
\[\Sigma^\ast(\ch_k(\gamma))=\ch_k(\gamma)\otimes 1+1\otimes \ch_k(\gamma)\,.\]
Under the realization map $\BD^X\to H^\ast(\FM_X)$, $\Sigma^\ast$ corresponds to pullback along the direct sum map $\Sigma\colon \FM_X\times \FM_X\to \FM_X$, hence the notation.
\item $\bR_{-1}$ is the derivation described in Section \ref{subsubsec: weight0}. Under the realization map, the operator $e^{z\bR_{-1}}$ corresponds to pullback along the $B\BG_m$ action $B\BG_m\times \FM_X\to \FM_X$, where $z$ is regarded as the generator of $H^2(B\BG_m)$.
\item $c_k(\Theta)\in \BD^X\otimes \BD^X$ are explicitly defined by the formula
\begin{equation}\label{eq: chernclassesTheta}
    \sum_{k\geq 0}c_k(\Theta)z^k=\exp\left(\sum_{\substack{a,b,\ell \geq 0\\\ell\equiv d \mod 2}}2(-1)^b(a+b+\ell-d-1)!z^{a+b+\ell-d}\ch_a\ch_b\big(\Delta_\ast \td_\ell(X)\big)\right)\end{equation}
where $\Delta_\ast \td_\ell(X)$ is the pushforward of the Todd class along the diagonal $\Delta\colon X\to X\times X$ and 
\[\ch_a\ch_b(\gamma^L\otimes \gamma^R)\coloneqq \ch_a(\gamma^L)\otimes \ch_b(\gamma^R)\in \BD^X\otimes \BD^X\,.\]
By Grothendieck--Riemann--Roch and Newton's identities (see, for example, the proof of Theorem 4.7 in \cite{blm}), the realization of the class $c_k(\Theta)$ in $H^\ast(\FM_X\times \FM_X)$ is the $k$-th Chern class of the complex
\[\Theta\coloneqq \Ext_{12}^\vee+\Ext_{21}\]
on $\FM_X\times \FM_X$, hence the notation.
\end{enumerate}

By construction, there exists a natural pairing
\[\langle -, -\rangle\colon \VX\otimes \BD^X\to \BQ\,.\]
If $A\in \bV_\alpha, B\in \bV_\beta$ and $D\in \BD^X$ then $[A,B]\in \bV_{\alpha+\beta}$ is determined by
\[\big\langle [A, B], D\big\rangle=(-1)^{\chi(\alpha, \beta)}\langle A\otimes B, \Delta_{\chi_\sym(\alpha, \beta)}D\rangle\]
where
\[\chi(\alpha, \beta)=\int_{X}\alpha^\vee\cdot \beta\cdot \td(X)\quad \textup{ and }\quad \chi_\sym(\alpha, \beta)=\chi(\alpha, \beta)+\chi(\beta, \alpha)\,.\]
Above, $\alpha^\vee=\sum_{j=0}^d (-1)^j \alpha_j$ where $\alpha_j$ is the component of $\alpha$ in $H^{2j}(X)$.

Borcherds shows in \cite{borcherds} that the 0-th product descends to a well-defined Lie bracket 
\[[-,-]\colon \VXhat \otimes\VXhat\to\VXhat\,.\]
It is observed in \cite[Lemma 3.12]{blm} that the $0$-th product also descends to a well-defined map
\begin{equation} \label{eq: partialift}[-,-]\colon \VXhat\otimes \VX\to \VX\,.\end{equation}
We will use the same notation for all these 3 maps, and loosely refer to them as the Lie bracket, despite technically only the one fully defined on $\VXhat$ being a Lie bracket.

There are counterparts of these Lie brackets if we replace $\VX$ by $H_\ast(\FM_X)$ and $\VXhat$ by $H_\ast(\FM_X^\rig)$, where $\FM_X^\rig$ is the rigidification of $\FM_X$ with respect to the $B\BG_m$ above, and we have a commutative square
\begin{center}
    \begin{tikzcd}
        H_\ast(\FM_X)\arrow[r] \arrow[d, two heads]&\VX \arrow[d, two heads]\\
        H_\ast(\FM_X^\rig)\arrow[r] &\VXhat
    \end{tikzcd}
\end{center}
where the top arrow is a homomorphism of vertex algebras, which in particular means that it is compatible with $[-,-]$, and the bottom arrow is a homomorphism of Lie algebras.

\begin{remark}
    The partial lift \eqref{eq: partialift} is dual to 
   \[ \Delta_{\chi_{\sym}(\alpha, \beta)}\colon \BD_{\alpha+\beta}^X\to \BD_{\inv, \alpha}^X\otimes \BD_{\beta}^X\,.\]
   The fact that  $\Delta_{\chi_{\sym}(\alpha, \beta)}$ lands in $\BD_{\inv, \alpha}^X\otimes \BD_{\beta}^X$ is equivalent to the identity 
   \[(\bR_{-1}\otimes \id)\circ \Delta_{\chi_{\sym}(\alpha, \beta)}=0\,.\]
   This can be argued directly, without invoking the fact that $\VX$ is a vertex algebra, by using \cite[Lemma A.2]{klmp}, which implies
   \[(\bR_{-1}\otimes \id)c_j(\Theta)=(\chi_\sym(\alpha, \beta)-j+1)c_{j-1}(\Theta)\,.\]
\end{remark}

\subsection{Twisting by line bundle on $X$}
\label{subsec: twistinglinebundle}
We review here the effect on descendents of twisting a moduli space of sheaves/complexes by a line bundle $H\in \Pic(X)$, following \cite[Section~2.7]{blm}. 

\begin{definition}\label{def: twistinglinebundle}
    Given a line bundle $H\in \Pic(X)$ we define the algebra homomorphism $T_H^\ast\colon \BD^X\to \BD^X$ by its value on generators:
    \[T_H^\ast(\ch_k(\gamma))=\sum_{j=0}^d \frac{1}{j!}\ch_{k-j}(\gamma \cdot c_1(H)^j)\,.\]
\end{definition}

The meaning and notation of $T_H^\ast$ is explained by the following observation. Let $T_H\colon \FM_X\to \FM_X$ be the automorphism given by tensoring $-\otimes H$. Then the diagram
\begin{center}
\begin{tikzcd}
\BD^X\arrow[r] \arrow[d, "T_H^\ast"]& H^\ast(\FM_X)\arrow[d, "T_H^\ast"]\\
\BD^X\arrow[r] & H^\ast(\FM_X)
\end{tikzcd}
\end{center}
commutes. It is easy to see that $T_H^\ast$ descends to an isomorphism $\BD^X_{\alpha'}\to \BD^X_\alpha$ or $\BD^X_{\inv,\alpha'}\to \BD^X_{\inv, \alpha}$ where $\alpha'=\alpha\cdot e^{c_1(H)}$.

We will also denote by $(T_H)_\ast$ the dual operators
\[(T_H)_\ast\colon H_\ast(\FM_X)\to H_\ast(\FM_X)\,,\quad (T_H)_\ast\colon \VX\to \VX\,,\quad (T_H)_\ast\colon \VXhat\to \VXhat\,.\]
Although we will not need it, it is not hard to show that $(T_H)_\ast$ is an automorphism of vertex/Lie algebras.

\subsection{Duality and $\BZ/2$ grading}\label{subsec: duality}
We define a $\BZ/2$-grading on the descendent algebra that plays a role in the formulation of the symmetry of the $\PT$ generating series. 

\begin{definition}\label{def: deltaZ2grading}
Let $\delta^\ast\colon \BD^X\to \BD^X$ be the algebra homomorphism sending 
\[\delta^\ast(\ch_k(\gamma))=(-1)^k\ch_k(\gamma)\,.\]
We say $D$ is even ($|D|=0$ in $\BZ/2$) or odd ($|D|=1$ in $\BZ/2$) if
\[\delta^\ast(D)=(-1)^{|D|} D\,.\]
\end{definition}
In other words, the even/odd part of $\BD^X$ is spanned by products $\prod_{i=1}^n \ch_{k_i}(\gamma_i)$ with $\sum_{i=1}^l k_i$ being even/odd. Note that this is not the $\BZ/2$ grading induced by the cohomological grading on $\BD^X$.

There is a natural interpretation of the meaning of $\delta^\ast$ in terms of the realization map. Let $\delta\colon \FM_X\to \FM_X$ send a complex $F$ to $F^\vee[2]$, where $(-)^\vee=R\mathcal Hom(-, \CO_X)$ is the derived dual. It follows from the fact that
\[(\delta\times \id)^\ast \CF=\CF^\vee[2]\]
in $\FM_X\times X$ that the diagram
\begin{center}
    \begin{tikzcd}
    \BD^X\arrow[d,"\delta^\ast"]\arrow[r] & H^\ast(\FM_X)\arrow[d, "\delta^\ast"]\\
    \BD^X \arrow[r] & H^\ast(\FM_X)
    \end{tikzcd}
\end{center}
commutes, where $\delta^\ast$ on the left is the operator formally defined in Definition \ref{def: deltaZ2grading} and the one on the right is the pullback along the map $\delta$. 

It is easy to see that $\delta^\ast$ descends to $\BD^X_{\alpha}\to \BD^X_{\alpha^\vee}$. Moreover, since $\bR_{-1}\circ \delta=-\delta\circ \bR_{-1}$, $\delta^\ast$ preserves the kernel of $\bR_{-1}$, and thus induces a morphism $\delta^\ast\colon \BD^X_{\inv, \alpha}\to \BD^X_{\inv, \alpha^\vee}$.

We will also denote by 
\[\delta_\ast \colon H_\ast(\FM_X)\to H_\ast(\FM_X)\,,\quad \delta_\ast \colon \VX\to \VX \,,\quad \delta_\ast \colon \VXhat\to \VXhat\]
the dual morphisms.

The next lemma establishes the compatibility between this $\BZ/2$ grading and wall-crossing formulas; it will be used in the proof of the $q\leftrightarrow q^{-1}$ symmetry part of Conjecture \ref{conj: rationality}. We endow $\BD^X\otimes \BD^X$ with the tensor $\BZ/2$ grading.

\begin{lemma}\label{lem: parityWC}
The operator $\Delta_s\colon \BD^X\otimes \BD^X\to \BD^X$ has parity\footnote{By this, we mean that it preserves the even/odd subspaces if $s+1\equiv 0 \mod 2$ and it sends even to odd and odd to even if $s+1\equiv 1 \mod 2$.} $s+1$.
\end{lemma}
\begin{proof}
From the definition it is straightforward that $\bR_{-1}$ has parity 1 and $\Sigma^\ast$ has parity 0. Moreover, it follows from \eqref{eq: chernclassesTheta} that 
\[\sum_{k\geq 0}\delta^\ast c_k(\Theta)z^k=\sum_{k\geq 0}c_k(\Theta)(-z)^k\]
(note that the sum specifies $\ell\equiv d \mod 2$), so $c_k(\Theta)$ has parity $k$. The claim now follows from the definition of $\Delta_s$. \qedhere
\end{proof}
\begin{remark}
A more geometric way to understand the parity of $c_k(\Theta)$, after taking geometric realization, is to observe that the isomorphism
\[\mathrm{RHom}_X(F_1^\vee[2], F_2^\vee[2])\simeq \mathrm{RHom}_X(F_2, F_1) \]
means that 
\[\delta^\ast \Ext_{12}=\Ext_{21}\,,\quad \delta^\ast \Ext_{21}=\Ext_{12}\]
and hence $\delta^\ast \Theta=\Theta^\vee$. One may think of \cite[Lemma 5.7]{bridgelandDTPT} as the Calabi--Yau counterpart to Lemma \ref{lem: parityWC}.
\end{remark}

\section{Wall-crossing between $\PT$ and $\sL$}
\label{sec: wallcrossing}
The main geometric input needed for the proof of our result is the wall-crossing path that Toda used to prove the Calabi--Yau case of the rationality and symmetry conjecture \cite{todaPTrationality}. In this section we recall Toda's stability condition and its properties. We then define $\sL$ invariants which are ``self-dual'' and are related to $\PT$ invariants via a wall-crossing formula.

\subsection{Toda's stability condition}
\label{subsec: todastability}
Let $\CA$ be the heart of $D^b(X)$ obtained by tilting $\Coh(X)$ as follows:
\[\CA=\langle \Coh_{\leq 1}(X), \Coh_{\geq 2}(X)[1]\rangle\,.\]
Let $\CB'\subseteq \CA$ be the full subcategory of $\CA$ extension-generated by 0-dimensional sheaves and by shifted sheaves $F[1]$ where $F$ is a pure 2-dimensional sheaf, and let $\CB\subseteq \CA$ be the subcategory of objects $E\in \CA$ such that 
\[\Hom(E', E)=0\textup{ for all }E'\in \CB'\,.\] Then $\langle \CB', \CB\rangle=\CA$ forms a torsion pair -- this is $\langle \CA^p_1,\CA^p_{1/2}\rangle$, in the notation of \cite{todaPTrationality, todalimit}, see \cite[Lemma 2.16]{todalimit}. The category $\CB$ is closed under subobjects but not under quotients. Two useful facts about $\CB$ is that it contains $\CO_X[1]$ and
\[\CB\cap \Coh_{\leq 1}(X)=\{\textup{pure 1-dim sheaves}\}.\]

If $E, F$ are objects of $\CB$ we say that a morphism $E\to F$ is a strict monomorphism (resp. epimorphism) if it is so in $\CA$ and the cokernel (resp. kernel) is also an object of $\CB$. Since $\CB$ is closed for subobjects, the condition of the kernel being in $\CB$ for a strict epimorphism is redundant.

Throughout, we will fix a polarization of $X$ and let $\mu_H$ be the slope of a 1-dimensional sheaf~$F$:
\begin{equation}\label{eq: slope1dimension}
    \mu_H(F)=\frac{\ch_3(F)}{\ch_2(F)\cdot H}
\end{equation}

\begin{definition}\label{def: todastability}
    We say that an object $E$ of $\CB$ with $\ch(E)=(-1,0,\beta, m)$ is $\mu_s$-semistable if 
    \begin{enumerate}
        \item For any pure 1-dimensional sheaf $F$ and strict epimorphism $E\twoheadrightarrow F$ in $\CB$ we have
        \[\mu_H(F)\geq s\,.\]
        \item For any pure 1-dimensional sheaf $F$ and strict monomorphism $F\hookrightarrow E$ in $\CB$ we have
        \[\mu_H(F)\leq s\,.\]
    \end{enumerate}
\end{definition}

A different way to formulate this notion of stability is the following. Let $\mu_s(F)=\mu_H(F)$ for a pure 1-dimensional sheaf, and let $\mu_s(E)=s$ for an object in $\CB$ of type $(-1, 0, \beta, m)$. Then $\mu_s$ behaves like a weak stability condition in the sense of \cite[Definition 4.1]{JO06III}. If we have a short exact sequence
\[0\to E'\to E\to E''\to 0\]
in $\CB$ with $E$ of type $(-1,0,\beta, m)$ then either $E'$ or $E''$ is a 1-dimensional sheaf by \cite[Lemma 3.16]{todaPTrationality}, and the other one also has topological type of the form $(-1, 0, \beta', m')$. So $\mu_s$-semistability is equivalent to $\mu_s(E')\leq \mu_s(E'')$ for any such short exact sequence. In the same way, a pure 1-dimensional sheaf is $\mu_s$-semistable if and only if it is $\mu_H$-semistable. 

\begin{remark}
Yet another way of formulating this definition is as a so-called ``limit stability'', defined via some central charge, as in \cite{todaPTrationality}. In this approach, the definitions above become theorems. 
\end{remark}

We let $\FL_{\beta, m}^s\subseteq \FM_X$ be the moduli stack of $\mu_s$-semistable objects in $\CB$ of topological type $(-1, 0, \beta, m)$, where $\FM_X$ is the Toën--Vaquié stack of perfect complexes on $X$. The following theorem summarizes several properties proved by Toda.

\begin{theorem}[{\cite{todaPTrationality}}]\label{thm: propertiesLspaces}\,
Let $X$ be a smooth projective 3-fold.
\begin{enumerate}
\item The stack $\FL_{\beta, m}^s$ is finite type and it is open in $\FM_X$.

\item There exists a constant $C_{\beta, m}$ such that the $\mu_s$-semistable objects with trivial determinant for $s>C_{\beta, m}$ are precisely Pandharipande--Thomas pairs. Assuming that $H^{>0}(\CO_X)=0$ we have
\[\FL_{\beta, m}^s\simeq P_{\beta, m}\times B\BG_m\,.\]

\item Given fixed $\beta$ and $s$, there are only finitely many $m$ such that the stack $\FL_{\beta, m}^s$ is non-empty.

\item An object $E\in \CB$ of topological type $(-1, 0, \beta, m)$ is $\mu_s$-semistable if and only if $\delta(E)$ is $\mu_{-s}$-semistable. Hence, we have an isomorphism of stacks
\[\delta\colon \FL_{\beta, m}^s\simeq \FL_{\beta, -m}^{-s}\,.\]
\end{enumerate}
\end{theorem}
\begin{proof}
Unless stated otherwise, the theorem numbers below refer to \cite{todaPTrationality}. The $\mu_s$-semistable objects in $\CB$ of type $(-1,0,\beta, m)$ are the same as the $\mu_{-sH/2+iH}$-limit semistable objects by Proposition 3.14. The statement (1) is Proposition 3.17; note that loc. cit. states that $\FL_{\beta, m}^s$ is open inside the stack of gluable complexes \cite{lieblich}, which is itself open in $\FM_X$ by the proof of \cite[Corollary 3.21]{TV07}. Statement (2) is Theorem 3.21, (3) is Lemma 4.4 and (4) is Lemma~4.3.
\end{proof}

Since $\FM_X$ has a natural structure of a derived stack, the open embedding $\FL_{\beta, m}^s\subseteq \FM_X$ endows the former with a derived structure as well. Under the assumption that $H^{>0}(\CO_X)=0$ the standard obstruction theory of $P_{\beta, m}$ \cite{PT} matches the obstruction theory coming from the derived enhancement induced by the open embedding $P_{\beta, m}\subseteq \FM_X^\rig$, so the isomorphism in (2) is also an isomorphism of derived stacks.

\begin{remark}
Theorem \ref{thm: propertiesLspaces}.2 says in particular that the notion of $\mu_s$-stability does not depend on the polarization $H$ for $s\gg 0$. The same is true for $s=0$, since $\mu_H(F)\geq 0$ is simply equivalent to $\ch_3(F)\geq 0$, which does not involve $H$. For general values of $s$, however, the stack $\FL_{\beta, m}^s$ does depend on $H$, but we will suppress $H$ from the notation.
\end{remark}

\subsection{Generalized homological invariants and wall-crossing}\label{subsec: generalizedinvariants}

We assume now that $X$ satisfies the conditions of Theorem \ref{thm: main}. Let $M_{\beta, m}$ be the moduli of $\mu_H$-semistable 1-dimensional sheaves on $X$. When $M_{\beta, m}$ has no strictly semistable objects, this moduli space is projective and quasi-smooth (or, for less derived readers, it admits a 2-term perfect obstruction theory), and therefore it defines an element $\sM_{\beta, m}\in H_\ast(\FM_X^\rig)$, or $\sM_{\beta, m}\in \VXhat$, as explained in Section~\ref{subsubsec: classesmoduliVA}. We recall that the element $\sM_{\beta, m}\in \VXhat$ essentially carries the information of how to integrate (weight 0) tautological classes against the virtual fundamental class of $M_{\beta, m}$.

In \cite[Theorem 7.68]{joyce} Joyce defines classes $\sM_{\beta, m}$ even in the presence of semistable sheaves. These are the ``generalized homological invariants'' referred to in the statement of Theorem~\ref{thm: wcintro}. 

Note that if $s$ is not in the discrete set 
\[\left\{\frac{m}{\beta'\cdot H}\colon  m\in \BZ\,,\, 0<\beta'\leq \beta\right\}\]
then $\FL_{\beta, m}^s$ does not have strictly semistable objects. As we will show in Section \ref{sec: technicalstacks}, when $\FL_{\beta, m}^s$ does not contain strictly semistable objects, its good moduli space $L_{\beta, m}^s$ is proper and quasi-smooth; therefore we can also integrate tautological classes against $[L_{\beta, m}^s]^\vir$ in that setting.

The setup of \cite{joyce} does not give us a way to construct generalized homological invariants $\sL_{\beta, m}^s$ in the presence of strictly semistables, since we do not know how to construct a framing functor in this setting. In \cite{KM} the authors construct ``generalized $K$-theoretic invariants'' without assuming the existence of framings. Since here we want (co)homological invariants, we will formally define $\sL_{\beta, m}^s$ by imposing the expected wall-crossing formula and then we will use the technology of \cite{KM} to show the desired properties of these invariants. Concretely, we define $\sL_{\beta, m}^s$ by
\begin{equation}\label{eq: Linvs}
    \sL_{\beta, m}^s\coloneqq \sum_{\substack{\beta_0+\ldots+\beta_k=\beta\\
    m_0+\ldots+m_k=m}}\tilde U(-; \mu_\infty, \mu_s) \cdot\big[\mathsf{M}_{\beta_k, m_k},\big[\ldots,\big[\mathsf{M}_{\beta_1, m_1},\mathsf{\PT}_{\beta_0, m_0}\big]\big]\ldots \big]\big]\in \VX
    \end{equation}
where the coefficients $\tilde U(-; \mu_\infty, \mu_s)$ are the combinatorial coefficients from \cite[Section 3.2]{joyce} and where $\mu_\infty$ denotes $\mu_{t}$ for $t\gg 0$; we will describe these coefficients explicitly when $s=0$ in Section \ref{subsec: wccoefficients}. The classes $\PT_{\beta_0, m_0}\in \VX$ in the formula correspond to integration against $[P_{\beta_0, m_0}]^\vir$ using the universal stable pair, cf. \eqref{eq: PTfunctional}. Note that in this formula we are using the partial lift of the Lie bracket \eqref{eq: partialift}. Note that \eqref{eq: Linvs} can also be interpreted as defining a class in $H_\ast(\FM_X)$.  The next theorem justifies calling these ``generalized homological invariants''.

\begin{theorem}\label{thm: Linvariants}
    Assume that $X$ satisfies the hypothesis of Theorem \ref{thm: main}. If $s$ is such that $\FL_{\beta, m}^s$ does not have strictly semistable objects then the class $\sL_{\beta, m}\in \VX$ defined by \eqref{eq: Linvs} agrees with integration against the virtual fundamental class, i.e.
    \[\langle \sL^s_{\beta, m}, D\rangle=\int_{[L_{\beta, m}^s]^\vir}\xi_\CI(D)\]
    for any $D\in \BD^X_{(-1, 0, \beta, m)}$, where $\xi_\CI$ is the $\pt$-normalized realization map to $H^\ast(L_{\beta, m}^s)$.
\end{theorem}

We will give a proof of this theorem in Section \ref{subsec: Kthyproof} using \cite{KM}. A consequence of Theorem~\ref{thm: Linvariants} is that $\sL_{\beta, m}^s=0$ if $\FL_{\beta, m}^s$ is empty.

Recall from Theorem \ref{thm: propertiesLspaces}.4 that $\delta$ sends $\mu_s$-semistable objects to $\mu_{-s}$-semistable objects. We use the same $K$-theoretic techniques to show that this is reflected at the level of generalized homological invariants. Recall the definitions of $(T_H)_\ast$ and $\delta_\ast$ from Sections \ref{subsec: twistinglinebundle} and \ref{subsec: duality}.

\begin{proposition}\label{prop: symminvariants}
    The classes $\sM_{\beta, m}\in \VXhat$ and $\sL_{\beta, m}\in \VX$  satisfy 
    \[(T_H)_\ast \sM_{\beta, m}=\sM_{\beta, m+H\cdot \beta}\, , \quad \delta_\ast \sM_{\beta, m}=\sM_{\beta, -m} \quad\textup{and}\quad \delta_\ast \sL_{\beta, m}^s=\sL_{\beta, -m}^{-s}\, . \]
\end{proposition}

\begin{remark}
 The statements of Theorem \ref{thm: Linvariants} and Proposition \ref{prop: symminvariants} also make sense when we regard $\bL$ classes as lying in $H_\ast(\FM_X)$, rather than in $\VX$. While those statements are surely true, our current proof only works in $\VX$ since we will crucially use \cite[Proposition 8.13]{KM}, which says that $K$-theoretic invariants determine integrals of tautological classes.

 In many important cases (e.g., that of $\mathbb P^3$), however, one can immediately deduce this result from the fact that $H^{\ast}(\FM_X)$ is tautologically generated. This latter statement is a theorem of Gross, cf. \cite{Gr19}.
\end{remark}

\subsection{$K$-theoretic invariants and proofs of Theorem \ref{thm: Linvariants} and Proposition \ref{prop: symminvariants}}\label{subsec: Kthyproof}
Let us very briefly summarize the main ideas of \cite{KM} that we will need. The authors define there an associative algebra structure on
\[\BK_{\textup{top}}(X)\coloneqq K_\ast^{\textup{top}}(\FM_X^{\rig})\,\]
where $K_\ast^{\textup{top}}(-)$ is the topological $K$-homology of a stack, as defined in Appendix A of loc. cit. This should be thought as being some kind of (topological) $K$-theoretic analog of $H_\ast(\FM_X^\rig)$ or $\VXhat$, and \cite[Definition 4.8, Appendix A]{KM} constructs generalized $K$-theoretic invariants
\[\sM_{\beta, m}^K, \sL_{\beta, m}^{s, K}\in \BK_{\textup{top}}(X)\,,\]
which are the counterparts to $\sM_{\beta, m}, \sL_{\beta, m}^{s}$, provided that certain conditions are satisfied. In \cite{KM}, $\sL_{\beta, m}^{s, K}$ would have been denoted by $\varepsilon_{(-1, 0, \beta, m)}^{\mu_s, K_{\textup{top}}}$. Unlike in Joyce's work, the existence of framing functors is not necessary, and these conditions are satisfied for us with minor modifications; we explain now briefly what needs to be done, and dedicate Section \ref{sec: technicalstacks} to proving these technical results. 
\begin{enumerate}
    \item Although the category $\CB$ is not abelian, the moduli stack $\FM_{\CB}$ of objects in $\CB$ satisfies the conclusion of Definition/Proposition 2.1 in \cite{KM} since it is open in $\FM_X$ by \cite[Proposition 3.17]{todaPTrationality} and preserved by the direct sum map and the $B\BG_m$ action on $\FM_X$. In particular $\BK_{\textup{top}}(\CB)=K_\ast^{\textup{top}}(\FM_{\CB}^\rig)$ is defined. We will use the natural algebra homomorphism $\BK_{\textup{top}}(\CB)\to \BK_{\textup{top}}(X)$ defined as the pushforward along $\FM_{\CB}^\rig\to \FM_X^\rig$ to regard all our invariants in $\BK_{\textup{top}}(X)$.
    \item For Assumption 2.11 we consider the set of permissible classes $C(\CB)_{\textup{pe}}$ to consist of classes with Chern character either $(0,0,\beta, m)$ or $(-1,0,\beta, m)$. Parts (4) and (5) of Assumption 2.11 follow from \cite[Lemma 3.16]{todaPTrationality}.
    \item The fact that $\FL_{\beta, m}^s$ is open in $\FM_\CB$ and finite type is shown in \cite[Proposition 3.17]{todaPTrationality}.
    \item We prove in Theorem \ref{thm: quasismooth} that $\FL_{\beta, m}^s$ are quasi-smooth under the hypothesis of Theorem \ref{thm: main}.
    \item We show in Theorem \ref{thm: Thetastratification} that $\FL_{\beta, m}^s$ is the semistable loci of a pseudo $\Theta$-stratification whose strata correspond to the possible $\mu_s$-Harder--Narasimhan types of objects in $\CB$. 
    \item The stability conditions $\mu_s$ are not equivalent to additive stability conditions, at least in an obvious way. The requirement that stability conditions are additive is imposed in \cite{KM} with the sole purpose of invoking the results of \cite{AHLH} that guarantee the existence of proper good moduli spaces for moduli of semistable objects in abelian categories. Instead, we will prove that $\FL_{\beta, m}^s$ admits a proper good moduli space (cf. Theorem \ref{thm: goodmoduli}) by verifying directly the conditions necessary for the main theorem of \cite{AHLH} to apply.
\end{enumerate}

We define in \cite[Definition 8.2]{KM} what it means for a class in $H_\ast(\FM_X^\rig)$ to be a homological lift of a class in $\BK_{\textup{top}}(X)$; roughly speaking, it means that the two classes are related by a Riemann--Roch type formula. We claim that $\sM_{\beta, m}$, $\sL_{\beta, m}^s$ are homological lifts of $\sM_{\beta, m}^K$, $\sL_{\beta, m}^{s, K}$, respectively; in this statement we are considering the $\sL$ invariants in $H_\ast(\FM_X^\rig)$ via the quotient $H_\ast(\FM_X)\to H_\ast(\FM_X^\rig)$. For the $\sM$ classes this is a direct consequence of Theorem 8.8 in loc. cit., since we have defined $\sM_{\beta, m}$ via Joyce's framing functor construction. Since stable pairs do not have strictly semistables, \cite[Lemma 8.4 (4)]{KM} applies and says that $\PT_{\beta, m}$ is a homological lift of~$\PT_{\beta, m}^K$. The wall-crossing for the generalized $K$-theoretic invariants \cite[Theorem 5.15]{KM} shows that $\sL_{\beta, m}^{s,K}$ can be written in terms of $\sM_{\beta, m}^K$, $\PT_{\beta, m}^K$ by a formula that is entirely analogous to \eqref{eq: Linvs}, except that it uses the commutator in~$\BK_{\textup{top}}(X)$ instead of the Lie bracket on $H_\ast(\FM_X^\rig)$. Finally, by comparing the two wall-crossing formulas we conclude using \cite[Theorem 8.6]{KM} that $\sL_{\beta, m}^{s}$ is a homological lift of $\sL_{\beta, m}^{s,K}$.

\begin{proof}[Proof of Theorem \ref{thm: Linvariants}]
We assume that $\FL_{\beta, m}^s$ has no strictly semistables. We will first argue that the equality holds when $D$ is weight 0. We have shown before that $\sL_{\beta, m}^{s}$ is a homological lift of $\sL_{\beta, m}^{s,K}$. By using again \cite[Lemma 8.4.4]{KM} we also find that the image of $[L_{\beta, m}^s]^\vir$ in $H_\ast(\FM_X^\rig)$ is a homological lift of the same class $\sL_{\beta, m}^{s,K}\in K_\ast^{\textup{top}}(\FM_X^{\rig})$. Then \cite[Proposition A.3]{KM} implies that the integrals of tautological classes in $H^\ast(\FM^\rig)$ against $\sL_{\beta, m}^{s}$ and $[L_{\beta, m}^s]^\vir$ agree, which is to say that their images in $\VXhat$ are the same.\footnote{Note that the statement of \cite[Proposition A.3]{KM} is for the stack $\FM_X$, not $\FM_X^\rig$, but it is not hard to adapt the proof. In our case, we may use the fact that the $\BG_m$ gerbe $\FM_X\to \FM_X^\rig$ is trivial when restricted to connected components of complexes with rank $\pm 1$. Picking a trivialization induces a universal complex on $\FM_X^\rig\times X$, which we can use as in the proof of \cite[Proposition 8.13]{KM}.}

We claim now that $\sL_{\beta, m}^s$ is a $\pt$-normalized class, in the sense that $\ch_1(\pt)\cap \sL_{\beta, m}^s$. Indeed, it is easy to see that $\PT_{\beta, m}$ being $\pt$-normalized (cf. Example \ref{ex: ptnormalizedPT}) and $\ch_0(\pt)\cap \sM_{\beta, m}=\ch_1(\pt)\cap \sM_{\beta, m}=0$ imply that any iterated bracket 
\[\big[\mathsf{M}_{\beta_k, m_k},\big[\ldots,\big[\mathsf{M}_{\beta_1, m_1},\mathsf{\PT}_{\beta_0, m_0}\big]\ldots \big]\big]\]
is also $\pt$-normalized (see also the proof of \cite[Lemma 5.11.a)]{blm}). It follows for example from \cite[Proposition 1.13]{klmp} that two $\pt$-normalized functionals that agree on weight 0 descendents must agree on every descendent.\qedhere
\end{proof}

\begin{proof}[Proof of Proposition \ref{prop: symminvariants}]
Toda shows in \cite[Lemma 4.3]{todaPTrationality} that the automorphisms $\delta$ and $T_H$ of the stack $\FM_X$ restrict to the following isomorphisms between the stacks of semistable objects:
\[T_H\colon \FM_{\beta, m}\xrightarrow{\sim} \FM_{\beta, m+H\cdot\beta}\, , \quad  \delta\colon  \FM_{\beta, m}\xrightarrow{\sim} \FM_{\beta, -m} \quad\textup{and}\quad  \delta\colon  \FL_{\beta, m}^s\xrightarrow{\sim} \FL_{\beta, -m}^{-s} \,.\]
By \cite[Proposition 4.10]{KM} we have the $K$-theoretic counterpart of the statement we want:
\[(T_H)_\ast \sM^K_{\beta, m}=\sM^K_{\beta, m+H\cdot \beta}\, , \quad \delta_\ast \sM^K_{\beta, m}=\sM^K_{\beta, -m} \quad\textup{and}\quad \delta_\ast \sL_{\beta, m}^{s, K}=\sL^{-s,K}_{\beta, -m}\, . \]
Thus, for example $\delta_\ast \sL_{\beta, m}^{s}$ and $\sL^{-s}_{\beta, -m}$ are the homological lifts of the same class, and their equality can then be shown as in the proof of Thoerem \ref{thm: Linvariants} since the $\pt$-normalization condition is easily seen to be preserved by $\delta_\ast$.
\end{proof}

\subsection{Wall-crossing coefficients between $\mu_0$ and $\mu_\infty$}\label{subsec: wccoefficients}

We will now describe the wall-crossing coefficients relating $\mu_\infty$ and $\mu_0$, finishing the proof of Theorem \ref{thm: wcintro}. Recall Definition~\ref{def: omord} of $\omord$. 

Formula \eqref{eq: Linvs} can be formally inverted using the properties of the $\tilde U$ coefficients \cite[Theorem 3.11]{joyce}:

\begin{equation}\label{eq: PTintermsofL}
   \mathsf{\PT}_{\beta, m}=\sum_{\substack{\beta_0+\ldots+\beta_k=\beta\\
    m_0+\ldots+m_k=m}}\tilde U(-; \mu_s, \mu_\infty) \cdot\big[\mathsf{M}_{\beta_k, m_k},\big[\ldots,\big[\mathsf{M}_{\beta_1, m_1},\sL^s_{\beta_0, m_0}\big]\big]\ldots \big]\big]\in \VX
    \end{equation}
Toda calculates in \cite[Section 4.3-4.5]{todaPTrationality} the wall-crossing coefficients between $\mu_0$ and $\mu_\infty$ when $m_i>0$, which are the same as the wall-crossing coefficients between $\mu_\epsilon$ for $0<\epsilon \ll 1$ and $\mu_\infty$. See also \cite[Theorem 2.8]{bucurves}, which has exactly the same combinatorics done more directly. The upshot of that calculation is the identity
   \[\PT_{\beta, m}=\sum_{\substack{\beta_0+\ldots+\beta_k=\beta\\
    m_0+\ldots+m_k=m\\
    0<\frac{m_1}{\beta_1\cdot H}\leq \ldots\leq \frac{m_k}{\beta_k\cdot H}}}\omord\Big(\frac{m_1}{\beta_1\cdot H},\ldots, \frac{m_k}{\beta_k\cdot H}\Big)\big[\mathsf{M}_{\beta_k, m_k},\big[\ldots,\big[\mathsf{M}_{\beta_1, m_1},\mathsf{L}_{\beta_0, m_0}^\epsilon \big]\big]\ldots \big]\big]\,.\]
Note that this formula differs from Theorem \ref{thm: wcintro} in two ways: it uses $\sL^\epsilon$ invariants, instead of $\sL^0$, and the inequality $0<\frac{m_1}{\beta_1\cdot H}$ is strict. We may relate $\sL^\epsilon$ and $\sL^0$ invariants using again the compatibility of the wall-crossing coefficients \cite[Theorem 3.11]{joyce}. Indeed, by unraveling the definition of the coefficients and using the combinatorial identity \cite[(81)]{todaPTrationality} we find that
\[\sL_{\beta, m}^{\epsilon}=\sum_{\beta_0+\ldots+\beta_k=\beta}\frac{1}{(k+1)!}\big[\mathsf{M}_{\beta_k, m_k},\big[\ldots,\big[\mathsf{M}_{\beta_1, m_1},\mathsf{L}_{\beta_0, m_0}^0 \big]\big]\ldots \big]\big]\,.\]
Combining this with the $\PT/\sL^\epsilon$ wall crossing formulas we get the wall-crossing formula in Theorem \ref{thm: wcintro}.

\begin{remark}\label{rmk: Mclassesdontcommute}
Unlike in the Calabi--Yau case, the classes $\sM_{\beta, m}$ do not seem to commute when $\dim H^2(X)>1$, at least for any obvious reason. Let $\beta_1, \beta_2$ be two non-proportional curve classes and choose $\gamma\in H^2(X)$ such that $\int_X\gamma\cdot \beta_1=-\int_X\gamma\cdot \beta_2\neq 0$. It follows that $\ch_2(\gamma)=0$ in $\BD^X_{(\beta, m)}$ where $\beta_1+\beta_2=\beta, m_1+m_2=m$, so that $\ch_3(\gamma)\in \BD^X_{\inv, (\beta, m)}$. An explicit calculation shows that 
\[\langle [\sM_{\beta_1, m_1},\sM_{\beta_2, m_2}],\ch_3(\gamma)\rangle =2\left(\int_X \gamma\cdot \beta_1\right)\sum_{i\in I}\langle \sM_{\beta_1, m_1},\ch_2(\gamma_i^L)\rangle \langle \sM_{\beta_2, m_2},\ch_2(\gamma_i^R)\rangle\]
where $\sum_{i\in I}\gamma_i^L\otimes \gamma_i^R$ is the projection of $\Delta_\ast c_1(X)$ onto $H^4(X)\otimes H^4(X)$. We see no reason why such expression should be non-zero, although we have not tried to calculate any explicit example. 
\end{remark}

\section{Moduli stacks of complexes: good moduli spaces, $\Theta$-stratifications and quasi-smoothness}\label{sec: technicalstacks}

In this section we will establish some technical results regarding the stacks $\FL_{\beta, m}^s$ which have been used in Section \ref{sec: wallcrossing} to apply the wall-crossing machinery of \cite{KM}. 

We have a chain of open embeddings
\[\FL_{\beta, m}^s\subseteq \FL_{\beta, m}\subseteq \FM_\CB\subseteq \FM_{\CA}\subseteq \FM_X\,,\]
where $\FM_\CA, \FM_\CB$ are the stacks of objects in $\CA$ and $\CB$, respectively, $\FL_{\beta, m}$ is the substack of objects of $\CB$ with Chern character $(-1,0, \beta, m)$, and $\FL_{\beta, m}$ the $\mu_s$-semistable locus. Each of these embeddings being open is shown in \cite{todaPTrationality}.

\subsection{Stack of filtrations}\label{subsec: Filt}

In what follows we will work with the stack of filtrations $\Filt(\FX)=\textup{Maps}(\Theta, \FX)$ of $\FX$, where $\Theta=[\BA^1/\BG_m]$. We remind the reader that the $\BC$-points of $\mathrm{Filt}(\FM_\CA)$ correspond to filtrations 
\begin{equation}
 \label{eq: Bfiltration}E_\bullet\colon \quad 
0=E_0\hookrightarrow E_1\hookrightarrow \ldots \hookrightarrow E_\ell\end{equation}
where the arrows are monomorphisms in $\CA$ which are not isomorphisms, together with a choice of integer weights $w_1<w_2<\ldots <w_\ell$. There are two natural maps $\Filt(\FX)\to \FX$, given by evaluation at 1 and evaluation at 0. Evaluation at 1 sends $E_\bullet$ to $E=E_\ell$, while evaluation at 0 sends $E_\bullet$ to the associated graded \[\textup{gr}(E_\bullet)\coloneqq \bigoplus_{i=1}^\ell E_i/E_{i-1}\,.\] 
See \cite[Section 6.3]{HLstructure} for more details.

The $\BC$-points of $\mathrm{Filt}(\FM_\CB)$ can be described in the same way, as filtrations $E_\bullet$ where $E_i$ are objects of $\CB$ and the arrows are strict monomorphisms in $\CB$; we say that such $E_\bullet$ is a filtration in $\CB$. This claim is a consequence of \cite[Proposition 1.3.1(3)]{HLstructure}, which asserts that $\mathrm{Filt}(\FM_\CB)$ is open in $\Filt(\FM_\CA)$ and that $E_\bullet$ is a $\BC$-point of $\mathrm{Filt}(\FM_\CB)$ if and only if the associated graded $\textup{gr}(E_\bullet)$
is an object of $\CB$. But this is equivalent to every factor $E_i/E_{i-1}$ being in $\CB$, since $\CB$ is closed under subobjects.

Recall that $\CB$ is the torsion free part of a torsion pair $\CA=\langle \CB', \CB\rangle$. If we let $\CA^t=\langle \CB, \CB'[1]\rangle$ be the tilt of $\CA$ according to this torsion pair then $\CB$ is the intersection $\CA\cap \CA^t$, and hence $\FM_{\CB}$ is the intersection of the open substacks $\FM_\CA$ and $\FM_{\CA^t}$ inside $\FM_X$. The stack of filtrations $\Filt(\FM_{\CB})$ is the intersection of the substacks $\Filt(\FM_{\CA})$ and $\Filt(\FM_{\CA^t})$ inside $\Filt(\FM_X)$. 

\subsection{Valuative criteria}\label{subsec: valuativecritaria}
In what follows, $R$ is a DVR essentially of finite type over $\BC$. Recall that a DVR is essentially of finite type if it can be written as the localization of a finite type $\BC$-algebra with respect to a prime ideal (necessarily of codimension 1). Let $\pi$ be a uniformizer of $R$. Essential finite type implies that the residue field $R/\pi$ is isomorphic to $\BC$. We let $K=\Frac(R)$, so that $\Spec(K)\to \Spec(R)$ is the inclusion of the generic point of $R$.

\begin{enumerate}
    \item ($\Theta$-reductive) Let $\Theta_R=\Theta\times \Spec(R)$ and let $0\in \Theta_R$ be the unique closed point. A stack $\FX$ is $\Theta$-reductive with respect to $R$ if any map $\Theta_R\setminus 0\to \FX$ can be uniquely extended to $\Theta_R\to \FX$.
    \item ($S$-complete) Let 
    \[\mathrm{ST}_R=\Spec(R[s,t]/(st-\pi))/\BG_m\,,\]
    where $\BG_m$ acts with weights $(1,-1)$ on $s$ and $t$. Let $0\in \mathrm{ST}_R$ be the closed point given by setting $s=t=0$. A stack $\FX$ is $S$-complete with respect to $R$ if any map $\mathrm{ST}_R\setminus 0\to \FX$ can be uniquely extended to $\mathrm{ST}_R\to \FX$.
    \item A stack $\FX$ satisfies the valuative criterion for universal closedness if every map $\Spec(K)\to \FX$ can be extended (not necessarily uniquely) to $\Spec(R)\to \FX$. 
\end{enumerate}

The stacks of objects in abelian categories, such as $\FM_\CA$, are known to satisfy these 3 valuative criterion \cite[Lemmas 7.16, 7.17, 7.18]{AHLH}. The same holds for $\CB$:

\begin{proposition}\label{prop: MBvaluative}
    The stack $\FM_\CB$ is $\Theta$-reductive, $S$-complete, and satisfies the valuative criterion for universal closedness with respect to any essentially finite type DVR. 
\end{proposition}
\begin{proof}
For $\Theta$-reductive and $S$-complete this follows immediately from the same properties for $\FM_{\CA}$ and $\FM_{\CA^t}$, together with the observation that $\FM_{\CB}$ is the intersection $\FM_{\CA}\cap \FM_{\CA^t}$ of open substacks in $\FM_X$. 

For universal closedness we follow the proof of \cite[Lemma 7.18]{AHLH}. First of all, let us introduce the categories $\mathcal A_{qc}$ and $\mathcal A^t_{qc}$ as \textit{``quasi-coherent''} analogues of $\mathcal A$ and $\mathcal A^t$. For this, we note that both $\mathcal A$ and $\mathcal A^t$ are hearts of natural $t$-structures on $D^b(X)$; these $t$-structures (cf. \cite[Subsection 6.2]{HLstructure}) induce natural $t$-structures on $\operatorname{DQCoh}(X)$. Let $\mathcal A_{qc}$ and $\mathcal A^t_{qc}$, resp., be the two corresponding hearts. \textit{Loc. cit.} shows that $\mathcal A_{qc} = \operatorname{Ind}(\mathcal A)$ (and respectively for $\mathcal A^t)$. Indeed, this equality holds since (cf. \cite[Proposition 6.1.7]{HLstructure}) the category in question is compactly generated, has all filtered colimits, is locally Noetherian, and has $\mathcal A$ as the subcategory of its compact objects. It is elementary to see, moreover, that $\mathcal A_{qc}$ admits the same description as $\mathcal A$: it can be described as $\langle \operatorname{QCoh}_{\leq 1}(X), \operatorname{QCoh}_{\geq 2}(X)[1]\rangle$ (and \textit{mutatis mutandis} for $\mathcal A^t_{qc}$), see \cite[Theorem 5.2]{saorin}.
 
 One can now define $\mathcal B_{qc}$ as the intersection of $\mathcal A_{qc}$ and $\mathcal A^t_{qc}$. We claim that it is $\operatorname{Ind}$-completion of $\mathcal B$. Indeed, $\mathcal B_{qc}\subseteq \mathcal A_{qc}$ can be described by the conditions $\operatorname{Hom}(\mathcal O_x, F) = 0$, $\operatorname{Hom}(E, F) = 0$, on $F \in \operatorname{DQCoh(X)}$, for all closed points $x \in X$ and pure $2$-dimensional sheaves $E$ over $X$. Both of these properties are preserved under filtered colimits, so the only thing we have to prove is $\mathcal B_{qc} \subseteq \operatorname{Ind}(\mathcal B)$. But every object in $\operatorname{DQCoh}(X)$ is a filtered colimit of its finitely presented subobjects by \cite[Theorem 5.2]{saorin}. These subobjects lie in $\mathcal B$ by Lemma \ref{lem: closedsub} below.

 Now, as in \cite[Definition 7.3]{AHLH}, we introduce the categories\footnote{\cite{AHLH} denotes this object by $\mathcal A_R$.} $\mathcal A_{qc, R}$ and $\mathcal A^{t}_{qc, R}$  for a ring $R$ as follows: $\mathcal A_{qc, R}$ is the category of pairs $(E, \xi_E)$ where $E$ is an object of $\mathcal A_{qc}$, and $\xi_E$ is a morphism $R \to \operatorname{End}_{\mathcal A_{qc}}(E)$ (and similarly for $\mathcal A^t$). It is obvious that $\mathcal \mathcal B_{qc, R} = \mathcal A_{qc, R} \cap \mathcal A^{t}_{qc, R}$ admits a similar description in terms of objects of $\mathcal B_{qc}$ equipped with a morphism from $R$ to their endomorphism ring.

 Now the proof of the desired universal closedness follows as in  \cite[Lemma 7.18]{AHLH}. If $E\in \CB_{qc, K}=\CA_{qc, K}\cap \CA^t_{qc, K}$ then $j_\ast E\in \CB_{qc, R}=\CA_{qc, R} \cap \CA^t_{qc, R}$. The subobjects $F_\alpha$ in loc. cit. are also in $\CB_{qc, R}$, since $\CB_{qc, R}$ is closed under taking  subobjects in $\CA_{qc, R}$, by Lemma~\ref{lem: closedsub} below. 
\end{proof}

\begin{lemma}\label{lem: closedsub} The subcategory $\mathcal B_{qc}$ of $\operatorname{DQCoh}(X)$ is closed under taking subobjects in $\mathcal A_{qc}$ and under taking quotients in $\mathcal A^t_{qc}$.
\end{lemma}
\begin{proof}
 For both categories  this follows immediately from the definition of $\mathcal B_{qc}$ via the vanishing of $\operatorname{Hom}$ functors.
\end{proof}

We will now prove that the stacks of semistable objects $\FL_{\beta, m}^s$ are also $\Theta$-reductive and $S$-complete. There is a general technique to deduce such statements for the semistable loci from the corresponding statement for the ambient stack, see \cite[Proposition 6.14]{AHLH}. Unfortunately, their result is for $\Theta$-stratifications that come from a numerical invariant in the sense of \cite[Definition 4.1.1]{HLstructure}, and we do not know how to interpret our stability in that way since it is not equivalent to an additive stability condition (cf. \cite[Section 7.3]{AHLH}), at least in an obvious way. Instead, we imitate the strategy used in the proof of \cite[Proposition 6.14]{AHLH}. We prove a few elementary lemmas about semistablity before proceeding.

The following lemma is similar to the standard fact that the for a (strong) stability condition, in the sense of \cite[Definition 4.1]{JO06III}, the semistable objects of fixed slope form an abelian category, see for example \cite[Proposition 2.20]{BST}. The stability condition $\mu_s$ is a weak stability, so semistable objects with the same slope do not form an abelian category (note that every object of type $(-1, 0, \beta, m)$ has the same slope), but we have the following weaker result:

\begin{lemma}\label{lem: stricmor} Let $E\in \CB$ be an object of topological type $(-1,0, \beta, m)$ and $F$ a pure 1-dimensional sheaf with $\mu_H(F)=s$.
    \begin{enumerate}
        \item If $F\hookrightarrow E$ is a strict monomorphism in $\CB$ then $E$ is $\mu_s$-semistable if and only if $E/F$ and $F$ are are both $\mu_s$-semistable
        \item If $E\twoheadrightarrow F$ is a strict epimorphism in $\CB$ then $E$ is $\mu_s$-semistable if and only if $F$ and $\ker(E\to F)$ are $\mu_s$-semistable.
    \end{enumerate}
\end{lemma}
\begin{proof}\textit{``Only if'' direction}. Suppose that $E$ is $\mu_s$-semistable.

(1) Let $G$ be a pure 1-dimensional sheaf. If $G$ is a quotient of $E/F$ then it is also a quotient of $E$, hence $\mu_H(G)\geq s$. If $G=\tilde G/F$ is a subobject of $E/F$, then $\mu_H(\tilde G)\leq s$, and, since $\mu_H(F)=s$, it follows that $\mu_H(G)\leq s$ as well, so $E/F$ is semistable. 

(2) The proof follows the same strategy. The subobject case is immediate. If $G$ is a quotient of $K\coloneqq \ker(E\to F)$, then $\tilde G=E/\ker(K\to G)$ is a 1-dimensional sheaf (not necessarily pure) which is a quotient of $E$, so $\mu_H(\tilde G)\geq s$. Indeed, if $\tilde G$ is pure, this is obvious. If not, consider $\tilde G_0 = \tilde G/T$ where $T$ is the maximal $0$-dimensional subobject in $\tilde G$. Then $E \to \tilde G_0$ is an epimorphism, and, hence a strict epimorphism since $\mathcal B$ is closed under subobjects. Thus, $\mu_H(\tilde G_0) \geq s$, and an elementary calculation shows that, thus, $\mu_H(\tilde G) \geq s$ as well. We have a short exact sequence of 1-dimensional sheaves $0\to G\to \tilde G\to F\to 0$, so we conclude that $\mu_s(G)\geq s$.

\textit{``If'' direction.} Suppose that $F$ and $E/F$ (resp. $\ker(E\to F)$) are semistable.

(1) Let $A \hookrightarrow E$ be a strict monomorphism in $\mathcal B$ with $A$ a pure $1$-dimensional sheaf. The exact sequence $0 \to F \to E \to Q\coloneqq E/F \to 0$ induces $0 \to A_F \to A \to A_Q \to 0$: here $A_F$ is $\operatorname{Ker}(A \to E \to E/F)$. Now, $\mu_H(A_F) \leq s$ from the assumptions since it is pure $1$-dimensional if non-zero. Now, $A_Q$ is either pure $1$-dimensional or has some $0$-dimensional subobject. In the latter case $E/F$ has a non-trivial $0$-dimensional subobject as well; contradiction with the fact that $F \to E$ is a \textit{strict} monomorphism. Thus, $A_Q$ is purely $1$-dimensional, hence $\mu_H(A_Q) \leq s$, hence $\mu_H(A) \leq s.$

By the definition of $\mu_s$-semistability, it remains to check that no pure $1$-dimensional quotient $E \to A$ destabilizes $E$ as well. This goes analogously, but this time with $A_F = \operatorname{Im}(F \to E \to A)$, using the same trick with the maximal torsion-free quotient as above.

(2) This follows from the same arguments used in the other case.\qedhere
\end{proof}

\begin{remark}
An alternative proof can be obtained by using the seesaw property for Toda's ``phase'' function $\phi_{\sigma_m}^{\dag}$ from \cite{todaPTrationality}. We leave the details to the reader.
\end{remark}

The next lemma can be thought of as the analog to \cite[Lemma 6.15]{AHLH} in our context.



\begin{lemma}\label{lem: grEsemistable}
Let $E_\bullet$ be a filtration of $E\in \FL_{\beta, m}$ and let $f\colon \Theta\to \FL_{\beta, m}$ be the map corresponding to $E_\bullet$ (and an arbitrary choice of weights). Then the following are equivalent:
\begin{enumerate}
    \item $f$ factors through $\FL_{\beta, m}^s$;
    \item $\textup{gr}(E_\bullet)$ is $\mu_s$-semistable;
    \item The graded pieces $E_j/E_{j-1}$ are all $\mu_s$-semistable of the same slope $s$;
    \item $E$ is semistable and the graded pieces $E_j/E_{j-1}$ all have the same slope $s$.
\end{enumerate}
\end{lemma}
\begin{proof}
$(1)\Leftrightarrow (2)$ is immediate from \cite[Proposition 1.3.1(3)]{HLstructure}. For $(2)\Leftrightarrow (3)$ note that the direct sum of objects of  $\mathcal B$ is semistable if and only if each of these objects is semistable and they all have the same slope. 

For $(3)\Leftrightarrow (4)$ note that exactly one of the graded pieces in question has topological type $(-1, 0, \beta', m')$ for some $\beta'$ and $m'$, and the others are $1$-dimensional. If $E$ is semistable the semistability of these $1$-dimensional sheaves follows from Lemma~\ref{lem: stricmor}; the semistability of the $(-1, 0, \beta', m')$ piece then follows from the same statement. Conversely, if the graded pieces all have the same slope and are semistable if tollows again from Lemma~\ref{lem: stricmor} that $E$ is semistable.
\end{proof}

The following lemma will be used to deduce $S$-completeness for the semistable locus.

\begin{lemma}\label{lem: oppositefiltrations}
    Let $E, E'\in \FL_{\beta, m}^s$. Suppose that we have two filtrations
    \begin{align*}
        0&=E_\ell\hookrightarrow E_{\ell-1}\hookrightarrow \ldots\hookrightarrow  E_{1}\hookrightarrow  E_0=E\\
        0&=E'_0\hookrightarrow  E'_{1}\hookrightarrow  \ldots\hookrightarrow E_{\ell-1}'\hookrightarrow E_\ell'=E'
        \end{align*}
in $\CB$ which have the same graded factors in reverse order, i.e.
\[E_{j}/E_{j+1}\simeq E'_{j+1}/E'_{j}\,\quad\textup{for }0\leq j<\ell.\]
Then the graded factors $E_{j}/E_{j+1}$ are all $\mu_s$-semistable of slope $s$. 
\end{lemma}
\begin{proof}
We prove the statement by induction on $\ell$. The case $\ell=1$ is trivial, so suppose $\ell>1$. All but one of the graded factors are 1-dimensional sheaves, so at least one of $E_1'$ and $E_{\ell-1}$ is 1-dimenisonal. Without loss of generality assume $E'_1$ is. Since $E'_1\simeq E/E_1$ is a subobject of $E'$ and a quotient of $E$ it follows that $\mu_H(E'_1)=s$, and therefore $E'_1$, $E/E_1'$ and $E_{1}=\ker(E\to E/E_1\simeq E_1')$  are all semistable by Lemma \ref{lem: stricmor}. Using the induction hypothesis with $E_1$ in place of $E$ and $E'/E_1'$ in place of $E'$ we conclude the statement. \qedhere
\end{proof}

We can now finally prove $\Theta$-reductivity and $S$-completeness of the semistable locus. 

\begin{proposition}\label{prop: valuativestability}
    The stack $\FL_{\beta, m}^s$ is $\Theta$-reductive and $S$-complete with respect to any essentially finite type DVR.
\end{proposition}
\begin{proof}
($\Theta$-reductive) Since $\FM_\CB$ is $\Theta$-reductive by Proposition \ref{prop: MBvaluative} and $\FL_{\beta, m}$ is a union of connected components in $\FM_\CB$, any map $\Theta_R\setminus 0\to \FL_{\beta, m}^s\subseteq \FL_{\beta, m}$ extends uniquely to $\Theta_R\to \FL_{\beta, m}$. Since $\FL_{\beta, m}^s$ is open in $\FL_{\beta, m}$ we only need to show that the closed point $0\in \Theta_R$ is mapped to $\FL_{\beta, m}^s$. As $\FM_\CB$ is locally of finite type we can extend the map $\Theta_R\to \FL_{\beta, m}$ to a map $\Theta\times C\to \FL_{\beta, m}$ where $C$ is a smooth finite type curve over $\BC$ and $x\in C$ is such that $\CO_{C, x}\simeq R$; by shrinking $C$ if necessary we may assume that $\Theta\times C\setminus \{(0,x)\}$ maps to~$\FL_{\beta, m}^s$. Let $\CE_\bullet$
be the corresponding filtration with $\CE_i\in D^b(C\times X)$ flat over $C$. For any $\BC$-point $y\in C$ distinct from $x$ we know from Lemma \ref{lem: grEsemistable} that $\bigoplus_{i=1}^\ell \CE_{i+1, y}/\CE_{i, y}$ is $\mu_s$-semistable, and in particular $\CE_{i+1, y}/\CE_{i, y}$ all have slope~$s$; since the slope is deformation invariant the same is true for $y=x$. Since $\CE_{\ell, x}$ is semistable (because $(1,x)\in \Theta\times C$ is mapped to $\FL_{\beta, m}^s$ by hypothesis) we conclude from Lemma \ref{lem: grEsemistable} that $\bigoplus_{i=1}^\ell\CE_{i+1, x}/\CE_{i, x}$ is also semistable, and hence $(0, x)$ lands in $\FL_{\beta, m}^s$ as well.

($S$-complete) As for $\Theta$-reductivity, a map $\mathrm{ST}_R\setminus 0\to \FL_{\beta, m}^s\subseteq \FM_\CB$ extends uniquely to $\mathrm{ST}_R\to \FL_{\beta, m}$ and we are left with showing that the closed point 0 is also mapped to $\FL_{\beta, m}^s$ for such extension. Recall that $\mathrm{ST}_R$ contains two copies of $\Theta$ as closed substacks, given by the equations $s=0$ and $t=0$:
\[\Theta_-=\Spec(R/\pi[t])/\BG_m\simeq \BA^1/\BG_m\quad \textup{ and }\quad \Theta_+=\Spec(R/\pi[s])/\BG_m\simeq \BA^1/\BG_m\,.\]
Note that $\BG_m$ acts with weight $-1$ in the definiton of $\Theta_-$ and with weight 1 in the definition of $\Theta_+$. Let us consider the filtrations associated to the restrictions of the map to $\Theta_-$ and $\Theta_+$:
\begin{align*}0=E_\ell\hookrightarrow E_{\ell-1}\hookrightarrow \ldots \hookrightarrow E_{1}\hookrightarrow E_0\\
0=E'_0\hookrightarrow E'_{1}\hookrightarrow \ldots \hookrightarrow E'_{\ell-1}\hookrightarrow E_\ell\end{align*}
Since $\Theta_-$ and $\Theta_+$ intersect in $0/\BG_m$ these filtrations have the same graded factors\footnote{This can also be seen from the explicit characterization of maps $\mathrm{ST}_R\to \FM_\CB$ in \cite[Corollary 7.14]{AHLH}. In the notation of loc. cit. the two filtrations that we call $\{E_i\}$ and $\{E_i'\}$ are $\{E_i/sE_{i-1}\}$ and $\{E_{i}/tE_{i+1}\}$, respectively. The graded factors in either case can be seen to be $E_{i}/(tE_{i+1}+sE_{i-1})$.}, i.e. $E_{i}/E_{i+1}\simeq E'_{i}/E'_{i-1}$, and we need to show that these graded factors are $\mu_s$-semistable of slope~$s$. Since $1_-\in \Theta_-$ and $1_+\in \Theta_+$ are mapped to $\FL_{\beta, m}^s$ we know that $E_0, E_\ell$ are $\mu_s$-semistable, and hence the claim follows from Lemma \ref{lem: oppositefiltrations}.
\end{proof}

\subsection{Pseudo $\Theta$-stratification}
\label{subsec: Thetastratification}

We will now show that $\FL_{\beta, m}$ admits a pseudo $\Theta$-stratification, in the sense of \cite{joyce}, whose semistable locus is $\FL_{\beta, m}^s$.

Given a filtration $E_\bullet$ in $\CB$ of $E\in \FL_{\beta, m}$ we call the tuple
\[\tau=\ch(E_\bullet)\coloneqq \big(\ch(E_1/E_0), \ch(E_2/E_1), \ldots ,\ch(E_\ell/E_{\ell-1})\big)\]
the topological type of $E_\bullet$. Every object of $\FL_{\beta, m}$ admits a $\mu_s$-Harder--Narasimhan filtration \cite[Lemma 3.9]{todaPTrationality}. We denote by $\textup{HN}_{\beta, m}(\mu_s)$ the set of possible Harder--Narasimhan types of objects in $\FL_{\beta, m}$ with respect to $\mu_s$; more precisely, elements of $\textup{HN}_{\beta, m}(\mu_s)$ have the form
\begin{equation}\label{eq: toptypefiltration}
\tau=\big((\beta_1, m_1), \ldots, (\beta_{i-1}, m_{i-1}), (-1, 0, \beta_i, m_i), (\beta_{i+1}, m_{i+1}), \ldots, (\beta_\ell, m_\ell)\big)
\end{equation}
with $\beta_j>0$ for $j\neq i$ and $\beta_i\geq 0$, $\sum_{i=1}^\ell \beta_i=\beta$, $\sum_{i=1}^\ell m_i=m$ and 
\[\frac{m_1}{\beta_1\cdot H}>\ldots>\frac{m_{i-1}}{\beta_{i-1}\cdot H}>s>\frac{m_{i+1}}{\beta_{i+1}\cdot H}>\ldots >\frac{m_{\ell}}{\beta_{\ell}\cdot H}\,.\]
One of the ingredients in the definition of pseudo $\Theta$-stratification is a partial order in $\textup{HN}_{\beta, m}(\mu_s)$. We construct a partial order inspired by Shatz' partial order on the set of Harder--Narasimhan types of vector bundles \cite{shatz}.

\begin{definition}\label{def: shatzline}
    Let $\tau$ be as in \eqref{eq: toptypefiltration}. We associate to $\tau$ the Shatz broken line of subobjects $\textup{SP}_1(\tau)$ and the Shatz broken line of quotients $\textup{SP}_2(\tau)$. The Shatz broken line of subobjects is defined as the line that connects the vertices 
    \[V_p=\left(\sum_{j=1}^p \beta_j\cdot H, \sum_{j=1}^p m_j \right)\in \BR^2\,,\quad p=0, 1, \ldots, i-1\]
    and is extended to $+\infty$ by adding the line of slope $s$
    \[V_{i-1}+t(1,s)\,,\quad t\geq 0\,.\]
    Similarly, the Shatz broken line of quotients is obtained by connecting the vertices  
    \[W_q=\left(\sum_{j=1}^q \beta_{\ell-j+1}\cdot H, \sum_{j=1}^q m_{\ell-j+1} \right)\in \BR^2\,,\quad q=0, 1, \ldots, \ell-i\,\]
and adding a line of slope $s$ at the end.

Given an object $E$ we denote by $\textup{SP}_j(E)$ the Shatz broken lines of the Harder--Narasimhan type of $E$.
\end{definition}

Note that if $\tau\in \textup{HN}_{\beta, m}(\mu_s)$ then $\textup{SP}_1(\tau)$ ($\textup{SP}_2(\tau)$) are the graphs of a piecewise linear concave (convex) function, i.e. the slopes of the segments are decreasing (increasing). 

\begin{definition}\label{def: partialorder}
 We define a partial $\leq $ order on $\textup{HN}_{\beta, m}(\mu_s)$ as follows: we say that $\tau\leq \tau'$ if $\textup{SP}_1(\tau)$ lies below $\textup{SP}_1(\tau')$ and $\textup{SP}_2(\tau)$ lies above $\textup{SP}_2(\tau')$ and, moreover, if $\textup{SP}_j(\tau)=\textup{SP}_j(\tau')$ for $j=1,2$ then $\tau=\tau'$. 
\end{definition}

\begin{lemma}\label{lem: semicontinuityHN}
    Let $F_\bullet$ be a filtration of $E$. Then $\ch(F_\bullet)\geq \ch(\textup{HN}(E))$. Moreover, if $\textup{SP}_j(F_\bullet)=\textup{SP}_j(\textup{HN}(E))$ for $j=1,2$ then $F_\bullet$ is the Harder--Narasimhan filtration of $E$. 
\end{lemma}
\begin{proof}
    To prove the inequality it is enough to show that $(\ch_2(F)\cdot H, \ch_3(F))$ is below $\textup{SP}_1(E)$ for any 1-dimensional sheaf $F$ which is a subobject of $E$ and above $\textup{SP}_2(E)$ for any 1-dimensional sheaf $F$ which is a quotient of $E$. We explain the statement for subobjects, since quotients are entirely analogous. The proof is an easy adaptation of \cite[Theorem 2]{shatz}, by induction on the length of the Harder--Narasimhan filtration $E_\bullet=\textup{HN}(E)$ of $E$. Let $E_1$ the be first object in the Harder. If $\rk(E_1)=-1$ that means that $E$ does not admit destabilizing 1-dimensional objects of slope $>s$ and that $\textup{SP}_1(E)$ is just a line of slope $s$, which makes it clear that $F$ is below $\textup{SP}_1(E)$. Otherwise, $E_1$ is a 1-dimensional sheaf and the induction step is exactly as in Shatz, using the induction hypothesis on $E_1\cap F\subseteq E_1$ and $E_1\vee F=(E_1\oplus F)/ (E_1\cap F)\subseteq E/E_1$. It is easy to see from the proof that if $(\ch_2(F)\cdot H, \ch_3(F))$ is a vertex of $\textup{SP}_1(E)$ then either $F=0$ or $E_1\subseteq F$, so by induction $F$ must be one of the objects appearing in the Harder--Narasimhan filtration of $E$, which shows the second part of the statement. 
    \end{proof}

The proof of the next theorem, which constructs a pseudo $\Theta$-stratification, follows closely the ideas in the proof of \cite[Theorem 2.2.2]{HLstructure}, with some of the simplifications that are also discussed there (e.g. $\Theta$-reductivity and quasicompact flag spaces). Since we are working with a pseudo $\Theta$-stratification instead of a regular $\Theta$-stratification, as in \cite{HLstructure}, some minor tweaks are necessary.

\begin{theorem}\label{thm: Thetastratification}
    There is a pseudo $\Theta$-stratification on $\FL_{\beta, m}$, indexed by the partially ordered set $(\textup{HN}_{\beta, m}(\mu_s), \leq)$, such that the semistable locus is $\FL_{\beta, m}^s$. 
\end{theorem}
\begin{proof}
To ease the notation let $\FL=\FL_{\beta, m}$. For $\tau\in \textup{HN}_{\beta, m}(\mu_s)$ let $\Filt(\FL)_\tau$ be the open and closes substack of $\Filt(\FL)$ parametrizing filtration of type $\tau$ and some arbitrary fixed choice of weights. 

We claim that the map forgetting the filtration $\Filt(\FL)_\tau\to \FL$ is proper for any $\tau\in  \textup{HN}_{\beta, m}(\mu_s)$. The fact that it satisfies the universal criterion for properness is equivalent to $\FL$ being $\Theta$-reductive, which we have shown in Proposition \ref{prop: MBvaluative}. It is also separated and locally finite type by \cite[Proposition 1.1.13]{HLstructure}. Moreover, \cite[Propisition 3.16]{todalimit} combined with the fact that the Quot scheme (with a fixed topological type) is quasicompact implies that the map $\Filt(\FL)_{\tau}\to \FL$ is also quasicompact, and therefore this map is proper.

We denote by $|\FX|$ the topological space of points of a stack $\FX$. We define the subset
\[|\FL|_{\leq \tau}=|\FL|\,\setminus \bigcup_{\tau'\nleq\tau} \textup{im}\left(|\Filt(\FL)_{\tau'}|\to |\FL|\right)\]
The $\BC$-points of $|\FL|_{\leq \tau}$ correspond to objects $E$ with $\textup{HN}(E)\leq \tau$: If $\textup{HN}(E)\leq \tau$ then $E$ cannot be in the image of 
$|\Filt(\FL)_{\tau'}|\to |\FL|$ for $\tau'\leq \tau$, since otherwise by Lemma \ref{lem: semicontinuityHN} we would have $\tau'\leq \textup{HN}(E)\leq \tau$, a contradiction. Conversely, if $\textup{HN}(E)\nleq \tau$ then clearly $E$ is not in $|\FL|_{\leq \tau}$.

We claim that $|\FL|_{\leq \tau}$ is open. Since $\FL$ is locally finite over $\BC$ it is enough to prove that $|U|\cap |\FL|_{\leq \tau}$ is open for any finite type open substack $U\subseteq \FL$. For finite type $U$ it follows from \cite[Proposition 3.16]{todalimit} that $|U|$ intersects $\textup{im}\hspace{-0.05cm}\left(|\Filt(\FL)_{\tau'}|\to |\FL|\right)$ only for finitely many $\tau'\in \textup{HN}_{\beta, m}(\mu_s)$.\footnote{This is closely related to the ``local finiteness'' condition in \cite{HLstructure}.}

By properness of the map $\Filt(\FL)_{\tau'}\to \FL$ it follows that $\textup{im}\big(|\Filt(\FL)_{\tau'}|\to |\FL|\big)$ is closed, and hence
\[|U|\cap |\FL|_{\leq \tau}=|U|\,\setminus \bigcup_{\tau'\nleq\tau}\left( |U|\cap \textup{im}\big(|\Filt(\FL)_{\tau'}|\to |\FL|\right)\big)\]
is open in $|U|$ since the union on the right hand side is finite.

Let us consider now the corresponding open substack $\FL_{\leq \tau}\subseteq \FL$ whose set of points is $|\FL|_{\leq \tau}$. Then $\Filt(\FL_{\leq \tau})$ is an open substack of $\Filt(\FL)$ and we define
\[\mathfrak S_\tau=\Filt(\FL)_\tau\cap \Filt(\FL_{\leq \tau})\,,\]
which is a union of connected components of $\Filt(\FL_{\leq \tau})$ since $\Filt(\FL)_\tau$ is a union of connected components of $\Filt(\FL)$. We will now show that $\mathfrak S_\tau$ is a $\Theta$ stratum in $\FL_{\leq \tau}$, i.e. $\textup{ev}_1\colon \mathfrak S_\tau\to \FL_{\leq \tau}$ is a closed embedding. By \cite[Corollary 2.1.9]{HLstructure} it is enough to show that $\mathfrak S_\tau$ is a weak $\Theta$-stratum, i.e. that the map above is radicial and finite. We have shown already that the map is proper, so it is enough to show it is radicial. By Lemma \ref{lem: radicial} below we just need to verify that the morphism is fully faithful on the groupoids of $\BC$-points; note that \cite[Proposition 1.1.13]{HLstructure} guarantees the technical assumptions necessary to apply Lemma \ref{lem: radicial}. 

Now the $\BC$-points of $\mathfrak S_\tau$ correspond, by \cite[Proposition 1.3.1 (3)]{HLstructure}, to filtrations $E_\bullet$ of topological type $\tau$ such that the Harder--Narasimhan type of $\textup{gr}(E_\bullet)\coloneqq \bigoplus_{i=1}^\ell E_i/E_{i-1}$ is $\leq \tau$. Note that $\textup{gr}(E_\bullet)$ comes itself with a filtration $\{\bigoplus_{i=1}^j E_i/E_{i-1}\}_j$ of type $\tau$, so by Lemma~\ref{lem: semicontinuityHN} this is the Harder--Narasimhan filtration of $\textup{gr}(E_\bullet)$, which implies that $E_\bullet$ is the Harder--Narasimhan filtration of $E=E_\ell$. The statement that $\textup{ev}_1\colon \mathfrak S_\tau\to \FL_{\leq \tau}$ induces a fully faithful functor at the level of $\BC$-points is now a consequence of the fact that Harder--Narasimhan filtrations are canonical and unique. 
\end{proof}

\begin{lemma}\label{lem: radicial}
    Let $\CS, \CX$ be algebraic stacks locally of finite type over $\BC$ and let $f\colon \CS\to \CX$ be a quasi-separated morphism representable by algebraic spaces. If the induced map $f(\BC)\colon \CS(\BC)\to \CX(\BC)$ on the groupoids of $\BC$-points is fully faithful then $f$ is radicial.
\end{lemma}
\begin{proof}
Let $X\to \CX$ be a smooth cover by an algebraic space and let $S=\CS\times_{\CX} X$, which is also an algebraic space by the representability of $f$. The condition that $f(\BC)$ is fully faithful implies that $f\colon S\to X$ is injective on $\BC$ points, so we have reduced ourselves to the case of algebraic spaces. 

By \cite[0484 and 040X]{stacks}, $S\to X$ being radicial is equivalent to surjectivity of the diagonal map $\Delta\colon S\to S\times_X S$, which is of finite presentation by the assumptions and \cite[0818]{stacks}. By Chevalley's theorem \cite[0ECX]{stacks} the image of $|\Delta|$ is a constructible subset of $|S\times_X S|$. The hypothesis that $f$ is injective on $\BC$ points means that every $\BC$-point on $|S\times_X S|$ is in the image of $\Delta$. Since $S\times_X S$ is a Jacobson algebraic space, a constructible subset that contains all the closed points is necessarily the whole space, otherwise the complement would contain a locally closed subset without $\BC$-points.
\end{proof}

We can finally conclude the existence of proper good moduli spaces by using the results from \cite{AHLH}.

\begin{theorem}\label{thm: goodmoduli}
    The stacks $\FL_{\beta, m}^s$ admit proper good moduli spaces.
\end{theorem}
\begin{proof}
By the main result of \cite{AHLH} this is equivalent to $\FL_{\beta, m}^s$ being $\Theta$-reductive, $S$-complete and satisfying the valuative criterion for universal closedness. By the same reasoning as in the proof of \cite[Theorem 7.23]{AHLH} it is enough to verify the 3 valuative criteria for DVRs essentially of finite type. The $\Theta$-reductive and $S$-complete statements are Proposition \ref{prop: valuativestability}. By \cite[Corollary 6.12]{AHLH} (see also \cite[Theorem 3.3.7]{joyce} for the adaptation to pseudo $\Theta$-stratifications) the valuative criteria for universal closedness follows from the existence of pseudo $\Theta$-stratification proven in Theorem \ref{thm: Thetastratification} and the result for the ambient stack $\FL_{\beta, m}$, which we proved in Proposition \ref{prop: MBvaluative}.\qedhere
\end{proof}

\subsection{Quasi-smoothness}\label{subsec: quasismooth}

The last point we want to address in this section is the quasi-smoothness of the stacks $\FL_{\beta,m}^s$. It is well-known that moduli stacks of semistable 1-dimensional sheaves supported in curve class $\beta$ are quasi-smooth as long as we have the positivity assumption in the statement of Theorem \ref{thm: main} holds. Indeed, quasi-smoothness is equivalent to $\Ext^3(F, F)=0$ for all semistable sheaves $F$, which is easily seen with Serre duality and the definition of stability. 

\begin{theorem}\label{thm: quasismooth}
    Suppose that $\beta'\cdot K_X<0$ for any $0< \beta'\leq \beta$ and $h^{3,0}(X)=0$. Then the moduli stack $\FL_{\beta, m}^s$ is quasi-smooth.
\end{theorem}
\begin{proof}
    We ought to show that for any $\mu_s$-semistable object $E$ with topological type $(-1,0,\beta, m)$ we have $\Ext^i(E,E)=0$ for $i>2$. By Serre duality this is equivalent to $\Ext^{3-i}(E, E\otimes K_X)=0$. Observe that tensoring by a line bundle preserves the heart $\CA$ and the subcategory $\CB$, and in particular both $E$ and $E\otimes K_X$ are elements of the heart $\CA$, so vanishing for $i>3$ is immediate and we only need to prove that $\Hom(E, E\otimes K_X)=0$. 

    Suppose we have a non-zero morphism $E\to E\otimes K_X$ and let $E_1, E_2$ be its kernel and cokernel:
        \begin{equation}\label{eq: kercokerproofquasismooth}
            0\to E_1\to E\to E\otimes K_X\to E_2\to 0
        \end{equation}
    Since $\CB$ is closed under subobjects the kernel $E_1$ is in $\CB$. Moreover, the quotient $E/E_1$ embeds into $E\otimes K_X$, which is in $\CB$, so $E_1\hookrightarrow E$ is a strict monomorphism in $\CB$. The cokernel $E_2$ can be decomposed according to the torsion pair $\langle \CB', \CB\rangle$ as 
    \begin{equation}\label{eq: BB'decomposition}
            0\to E_2''\to E_2\to E_2'\to 0
    \end{equation}
    with $E_2'\in \CB$ and $E_2''\in \CB'$; recall that $\rk(E_2'')=0$ for any $E_2''\in \CB'$. Then $E\otimes K_X\to E_2'$ is a strict epimorphism in $\CB$, and hence the same is true for $E\to E_2'\otimes K_{X}^{-1}$. Therefore, by \cite[Lemma 3.16]{todaPTrationality} we have one of the two possibilities:
\begin{enumerate}
    \item $\rk(E_1)=\rk(E_2)=\rk(E_2')=0$, in which case $E_1$ and $E_2'$ are pure 1-dimensional sheaves.
    \item $\rk(E_1)=\rk(E_2)=\rk(E_2')=-1$, in which case both $E_1$ and $E_2'\otimes K_X^{-1}$ have topological types of the form $(-1,0, \beta', m')$ for some $\beta'\leq \beta$. 
\end{enumerate}
Note that in either case the epimorphism $E\to E_2'\otimes K_{X}^{-1}$ in $\CB$ implies that $\ch_1(E_2'\otimes K_X^{-1})=0$. We start with the first case. We have
\[\ch_1(E_2'')\overset{\eqref{eq: BB'decomposition}}{=}\ch_1(E_2)\overset{\eqref{eq: kercokerproofquasismooth}}{=}\ch_1(E\otimes K_X)=-K_X\,.\]
Since $E_2''$ can be obtained by successive extensions of 0-dimensional sheaves and shifts $F[1]$ of pure 2-dimensional sheaves $F$, it would follow that $K_X$ was effective, contradicting the assumption that $h^{3,0}(X)=0$.  

For the second case we start with the observation that
\[\ch_1(E_2')=-K_X=\ch_1(E\otimes K_X)\overset{\eqref{eq: kercokerproofquasismooth}}{=}\ch_1(E_2)\]
and therefore $\ch_1(E_2'')=0$, which implies that $E_2''$ is a 0-dimensional sheaf. Let
\[Q_1=E/E_1\,,\quad Q_2=\ker(E\to E_2\otimes K_X^{-1})
,\quad Q'_2=\ker(E\to E'_2\otimes K_X^{-1})\,.\]
All of these have rank 0, and $Q_1, Q_2'$ are in $\CB$, as argued before, so $Q_1, Q_2'$ are both pure 1-dimensional sheaves. Moreover, we have a short exact sequence
\begin{equation}\label{eq: ses0dimsheaf}
0\to Q_2\to Q_2'\to E_2''\to 0\end{equation}
in $\CA$ and therefore it follows that $Q_2$ is also a pure 1-dimensional sheaf. By \eqref{eq: kercokerproofquasismooth} the sheaves $Q_1$ and $Q_2\otimes K_X$ have the same Chern characters. Now stability of $E$ implies that 
\[s\overset{(1)}{\geq} \mu_H(Q_2') \overset{(2)} \geq \mu_H(Q_2)=\mu_H(Q_1\otimes K_X^{-1})\overset{(3)} >\mu_H(Q_1)\overset{(4)}\geq s\,,\]
which is a constradiction. Here, (1), (4) follow from stability of $E$, (2) follows from \eqref{eq: ses0dimsheaf} and the fact that $E_2''$ is 0-dimensional, and (3) uses the positivity assumption for
\[\ch_3(Q_1\otimes K_X^{-1})=\ch_3(Q_1)-K_X\cdot \ch_2(Q_1)>\ch_3(Q_1)\,.\qedhere\]
\end{proof}

\begin{remark}
    The assumption that $h^{3,0}(X)=0$ is necessary since we are working with the standard obstruction theory governed by $\Ext^i$ instead of the traceless obstruction theory. For example, if $E$ is a stable pair then it is shown in \cite[Lemma 2.10]{PT} that $\tr\colon \Ext^3(E, E)\to H^3(\CO_X)$ is an isomorphism. It is conceivable that the traceless obstruction theory of the fixed determinant version of $\FL_{\beta, m}^s$ is quasi-smooth, although our current proof does not show that.
\end{remark}

\section{Proof of Theorem \ref{thm: main}}
\label{sec: proofmainthm}
We will now give a proof of our main theorem. The geometric input that we will need is Theorem \ref{thm: wcintro}. The next proposition is an easy consequence of the symmetries of $\sM$ invariants. We remind the reader of the $\BZ/2$ grading on the descendent algebra introduced in Definition \ref{def: deltaZ2grading}

\begin{proposition}\label{prop: Mquasipolynomial}
    Let $D\in \BD^X$. Then the function
    \[\BZ\ni m\mapsto \langle \mathsf{M}_{\beta, m/2}, D\rangle \in \BQ\]
    is a quasi-polynomial in $m$ of period dividing $2(\beta\cdot H)$. If $D$ is even/odd then this quasi-polynomial is even/odd, respectively.
\end{proposition}
\begin{proof}
   For the parity statement we have the following:
    \[\langle \mathsf{M}_{\beta, -m/2}, D\rangle=\langle \delta_\ast \mathsf{M}_{\beta, m/2}, D\rangle=\langle \mathsf{M}_{\beta, m/2}, \delta^\ast D\rangle=(-1)^{|D|}\langle \mathsf{M}_{\beta, m/2}, D\rangle\,.\]
    For the quasi-polynomiality we want to argue that the following is a polynomial in $k$:
     \[\langle \mathsf{M}_{\beta, m/2+k(\beta\cdot H)}, D\rangle=\langle (T_{kH})_\ast \mathsf{M}_{\beta, m/2}, D\rangle=\langle \mathsf{M}_{\beta, m/2}, T_{kH}^\ast D\rangle\,.\]
     This is immediate from the definition of $T_{kH}$ since 
     \[T_{kH}^\ast \left(\prod_{i=1}^n \ch_{k_i}(\gamma_i)\right)= \prod_{i=1}^n \Big(\ch_{k_i-1}(\gamma_i)+k\ch_{k_i}(\gamma_i\cdot H)+\frac{k^2}{2}\ch_{k_i-2}(\gamma_i\cdot H^2)\Big)\,\]
     is itself a polynomial in $k$ with coefficients in $\BD^X$.\qedhere 
\end{proof}

Consider now a fixed $\beta$ as in the statement of Theorem \ref{thm: main} and $D\in \BD^X$. We will analyze the formula that we get for $\langle \PT_{\beta, m},D\rangle$ by applying the wall-crossing formula of Theorem \ref{thm: wcintro}. The first remark is that given a fixed $\beta$ there are only finitely many decompositions 
\[\beta=\beta_0+\ldots+\beta_k\,\]
with $\beta_0\geq 0$ and $\beta_1, \ldots, \beta_k>0$. We will fix such a partition and consider its contribution to $\langle \PT_{\beta, m},D\rangle$. Note that we have:
\begin{align}\label{eq: eulerpairings}
\chi\big((0,0,\beta, m), (-1,0,\beta', m')\big)&=m-d_\beta/2\\
\chi\big((-1,0,\beta', m'), (0,0,\beta, m)\big)&=-m-d_\beta/2 \nonumber \\
\chi_\sym\big((0,0,\beta, m), (-1,0,\beta', m')\big)&=-d_\beta\,. \nonumber
\end{align}

To ease the notation we will denote $d_j\coloneqq d_{\beta_j}$. By the definition of $[-,-]$ we have
\begin{align*}
\Big\langle\big[\mathsf{M}_{\beta_k, m_k},\big[\ldots,\big[\mathsf{M}_{\beta_1, m_1},\mathsf{L}_{\beta_0, m_0}\big]\ldots \big]\big], D\Big\rangle=\sum_i \langle \sL_{\beta_0, m_0}, D_0^i\rangle \prod_{j=1}^k (-1)^{m_j-d_j/2}\langle \sM_{\beta_j, m_j}, D_j^i\rangle
\end{align*}
where 
\begin{align*}\sum_i \bigotimes_{j=0}^k &D_{k}^i\otimes D_{k-1}^i\otimes \ldots\otimes D_1^i\otimes D_0^i\coloneqq \\
&(\id\otimes\ldots\otimes \id\otimes \Delta_{d_{k}})\circ \ldots\circ(\id\otimes\Delta_{d_{2}})\circ\Delta_{d_{1}} D\in \big(\BD^X_\inv\big)^{\otimes k}\otimes \BD^X\,.
\end{align*}
Crucially, these $D_j^i$ depend only on the partition of $\beta$, but not on $m_0, m_1,\ldots, m_k$, which allows us to analyze the dependence on $m$. It follows from Lemma \ref{lem: parityWC} that for every $i$ we have
\begin{equation}
    \label{eq: sumofparities}
    \sum_{j=0}^k |D_j^i|\equiv |D|+\sum_{j=1}^k (d_j+1) \mod 2\,.
\end{equation}

Let
\[f_0^i(m)=\langle \sL_{\beta_0, m/2}, D_0^i\rangle\quad \textup{ and }\quad f_j^i(m)=(-1)^{m/2-d_j/2}\langle \sM_{\beta_j, m/2}, D_j^i\rangle\,\textup{ for }j=1,\ldots, k\,.\]
Note that $f_j^i(m)=0$ unless $m\equiv d_j \mod 2$, so in particular the sign $(-1)^{m/2-d_j/2}$ makes sense. 

With the notation set above, the contribution of the fixed partition $\beta=\beta_0+\ldots+\beta_k$ to $Z_\beta(q|D)$ is 

\[\sum_i \left(\sum_{m_0\in \frac{1}{2}\BZ}f_0^i(2m_0)q^{m_0}\right)\left(\sum_{\substack{
    m_1, \ldots, m_k\in \frac{1}{2}\BZ\\
    0\leq\frac{m_1}{\beta_1\cdot H}\leq \ldots\leq \frac{m_k}{\beta_k\cdot H}}}\omord\Big(\frac{m_1}{\beta_1\cdot H},\ldots, \frac{m_k}{\beta_k\cdot H}\Big) \prod_{j=1}^k f_j^i(2m_j)q^{m_j}\right)\,.\]
Let $A(q), B(q)$ be the first and second factors in the above formula. 

By Theorem \ref{thm: propertiesLspaces}.3 the moduli stacks $\FL_{\beta, m}^{0}$ are empty for all but finitely many $m$, and thus $\sL_{\beta, m}=0$ for all but finitely many $m$ by Theorem \ref{thm: Linvariants}. Since $\delta_\ast \sL_{\beta, m}=\sL_{\beta, -m}$, the same argument as in the proof of Proposition \ref{prop: Mquasipolynomial} shows that $m\mapsto f_0^i(m)$ has the same parity as~$|D_0^i|$. Thus, $A(q)$ is a Laurent polynomial\footnote{More precisely, $q^{d_0/2}A(q)$ is a Laurent polynomial in $q$. A similar caveat applies to all the other ``rational functions'' that we consider below.} satisfying the symmetry
\[A(q^{-1})=(-1)^{|D_0^i|}A(q)\,.\]

We now consider $B$. Note that the quasi-polynomial (of period 2)
\[\BZ\ni m\mapsto \begin{cases}
    (-1)^{m/2-d/2} &\textup{ if }m\equiv d\mod 2\\
    0 & \textup{ otherwise}
\end{cases}\]
has the same parity as $d$ (i.e. is even if $d$ is even and odd if $d$ is odd). Hence, by Proposition \ref{prop: Mquasipolynomial} it follows that $f_j^i(m)$ is a quasi-polynomial of period dividing $2(\beta\cdot H)$ and parity $|D_j^i|+d_j$. Thus, by Corollary \ref{cor: combinatorialsymm} we find that $B(q)$ is the expansion of a rational function satisfying the symmetry
\[B(q^{-1})=(-1)^{\sum_{j=1}^k(|D_j^i|+d_j+1)} B(q)\,. \]

Combining the symmetries for $A$ and $B$ with \eqref{eq: sumofparities} we find that $A(q)B(q)$ is the expansion of a rational function with symmetry
\[A(q^{-1})B(q^{-1})=(-1)^{|D|}A(q)B(q)\,.\]
Since we have exhibited $Z_\beta(q|D)$ as a finite sum of such expansions, the proof of Theorem \ref{thm: main} is now complete.

\section{Primary insertions and Gopakumar--Vafa strong rationality}
\label{sec: primaryGV}

Insertions that only use the tautological classes $\ch_2(\gamma)$ are called primary. As stated in Section \ref{subsec: primaryintro}, our main wall-crossing formula simplifies drastically in the primary regime. For notational convenience we will assume throughout this section that $X$ is Fano, although the results hold in the generality of Theorem \ref{thm: main} (see Remark \ref{rmk: positiveclassesprimary}).

\subsection{Wall-crossing for primary descendents}

The next lemma explains how the Lie brackets appearing in the general wall-crossing formula of Theorem \ref{thm: wcintro} simplify when we restrict our attention to primary descendents
\begin{equation}
    \label{eq: primaryD}D=\prod_{i=1}^n \ch_2(\gamma_i)\,.
\end{equation}

\begin{lemma}\label{lem: bracketprimary}
Let $D$ be a primary descendent. Let $\sM\in \VXhat$ be a class in the image of the map $H_\ast(\FM_{\beta, m}^\rig)\to H_\ast(\FM_X^\rig)\to \widecheck \bV_X$, where $\FM_{\beta, m}\subseteq \FM_X$ is the stack of 1-dimensional sheaves on $X$ with Chern character~$(\beta, m)$. Let $\mathsf{P}\in \bV_{(-1, 0, \beta', m')}$. We have
\begin{align}
    \langle \big[\sM, \mathsf{P} \big], D\rangle&=0\quad \nonumber \textup{ if }d_\beta>1\\
     \langle \big[\sM, \mathsf{P} \big], \label{eq: bracketprimarydescendents}D\rangle&=(-1)^{m-1/2}\sum_{I\sqcup J=[n]}\langle \sM, \prod_{i\in I}\ch_2(\gamma_i)\rangle \langle \mathsf{P}, \prod_{j\in J}\ch_2(\gamma_j)\rangle \quad \textup{ if }d_\beta=1\,,
\end{align}
where the last sum runs over partitions of $[n]= \{1, \ldots, n\}$.
\end{lemma}

\begin{remark}\label{rmk: cohdegreeconstraints}
    Note that the descendent $\ch_2(\gamma)$ has cohomological degree $\deg(\gamma)-2$. In particular, when studying primary insertions we may restrict ourselves to classes $\gamma_i\in H^{>2}(X)$. If we do so, the type of partitions that can appear in \eqref{eq: bracketprimarydescendents} is very restricted: since the virtual dimension of $\sM$ is 1, the only possibility for $I$ is $I=\{i\}$ with $\deg(\gamma_i)=4$ or $I=\{i_1, i_2\}$ with $\deg(\gamma_{i_1})=\deg(\gamma_{i_2})=3$.
\end{remark}

\begin{proof}
We recall that, by definition of the Lie bracket and \eqref{eq: eulerpairings} we have
\[\langle \big[\sM, \mathsf{P} \big], D\rangle=(-1)^{m-d_\beta/2}\langle \sM\otimes \mathsf P, \Delta_{-d_\beta} D\rangle\]
where 
\[\Delta_{-d_\beta}D=\sum_{j\geq 0}c_{-d_\beta+j+1}(\Theta)\circ \left(\frac{\bR_{-1}^j}{j!}\otimes \id\right)\circ \Sigma^\ast D\,.\]
Note first that
\[\Sigma^\ast D=\sum_{I\sqcup J=[n]} D_I\otimes D_J\]
where $D_I=\prod_{i\in I}\ch_2(\gamma_i)$. The key observation is the following: since the universal sheaf in $\FM_{\beta, m}\times X$ is supported in codimension 2, the realization of $\ch_0(\gamma), \ch_1(\gamma)$ in $H^\ast(\FM_{\beta, m})$ vanishes for any $\gamma$, and therefore the realization of $\bR_{-1}^j D_I$ vanishes for $j>0$. Since by hypothesis $\sM$ comes from $H_\ast(\FM_{\beta, m})$, we have $\langle \sM, D'\cdot  \bR_{-1}^j(D_I)\rangle=0$ for any $j>0$ and $D'\in \BD^X$. Thus
\[\langle \sM\otimes \mathsf P, \Delta_{-d_\beta}D\rangle=\langle \sM\otimes \mathsf P,\sum_{I\sqcup J=[n]}c_{1-d_\beta}(\Theta) D_I\otimes D_J\rangle\,.\qedhere\]
\end{proof}

\begin{lemma}\label{lem: easychiindependence}
Let $\beta\in H_2(X)$ be such that $d_\beta=1$ and let $D$ be a primary descendent. Then 
\[\int_{[M_{\beta, m}^H]^\vir}D\]
does not depend on $H$ and on $m$. 
\end{lemma}
\begin{proof}
    Since $-K_X$ is ample and $-K_X\cdot \beta=1$ it follows that $\beta$ is an irreducible curve class, i.e. it cannot be written as a sum of two effective classes. Therefore, a 1-dimensional sheaf $F$ with $\ch_2(F)=\beta$ is $\mu_H$-stable if and only if it is $\mu_H$-semistable if and only if it is pure. In particular, the moduli space $M_{\beta, m}^H$ does not depend on $H$. Moreover, $F\mapsto F\otimes K_X$ determines an isomorphism $M_{\beta, m}\simeq M_{\beta, m-1}$. Hence
    \[\int_{[M_{\beta, m-1}^H]^\vir} D=\int_{[M_{\beta, m}^H]^\vir} T_{K_X}^\ast (D)=\int_{[M_{\beta, m}^H]^\vir} D\,\]
    where the last equality uses again the fact that the realization of $\ch_0(\gamma), \ch_1(\gamma)$ in $M_{\beta, m}^H$ vanishes.
\end{proof}

\begin{question}
Are the primary $\sM$ invariants $\int_{\sM_{\beta, m}}D$ independent of $m$ for arbitrary $\beta$? The wall-crossing formula shows that they are polarization independent, and analogous phenomena have been observed for different moduli spaces in \cite[Proposition 1.8]{LMP2}, \cite[Theorem 8.2]{parabolic} and work in preparation by W. Lim, W. Pi and the second author on moduli of 1-dimensional sheaves on surfaces. 
\end{question}

In light of Lemma \ref{lem: easychiindependence}, when $d_\beta=1$ we will write $\int_{\sM_{\beta}}D=\int_{[M_{\beta, m}^H]^\vir}D$
for primary descendents $D$. To formulate the primary $\PT/\sL$ wall-crossing formula, it is convenient to encode primary invariants into generating series.  We use a formal Novikov variable $Q$ to keep track of the curve class $\beta$. Furthermore, we introduce variables to keep track of the primary insertions. For this, we fix a homogeneous basis $\{\gamma_a\}_{a\in A}$ of $H^{\ast}(X)$. We denote by $A_k, A_{>k}\subseteq A$ the subsets corresponding to elements $\gamma_a\in H^k(X), H^{>k}(X)$, respectively. Consider the formal variables $\bt=\{t_a\}_{a\in A_{>2}}$, regarded as dual to the basis $\{\gamma_a\}_{a\in A_{>2}}$ of $H^{>2}(X)$. We now introduce the following partition functions:

\begin{align}\label{eq: partitionfunctions}Z_{\textup{prim}}^{\PT}(Q, q, {\bf t})&=\sum_{\beta\geq 0}\sum_{m\in \frac{1}{2}\BZ}Q^\beta q^m \sum_{a_1, \ldots, a_n\in A_{>2}}\frac{t_{a_1}\ldots t_{a_n}}{n!}\int_{\PT_{\beta, m}}\prod_{i=1}^n\ch_2(\gamma_{a_i})\\ \nonumber
Z_{\textup{prim}}^{\sL}(Q, q, {\bf t})&=\sum_{\beta\geq 0}\sum_{m\in \frac{1}{2}\BZ}Q^\beta q^m \sum_{a_1, \ldots, a_n\in A_{>2}}\frac{t_{a_1}\ldots t_{a_n}}{n!}\int_{\sL_{\beta, m}}\prod_{i=1}^n\ch_2(\gamma_{a_i})\\ \nonumber
Z_{\textup{prim}}^{\mathsf{M}}(Q, {\bf t})&=\sum_{\beta\textup{ s.t. }d_\beta=1}Q^\beta \sum_{a_1, \ldots, a_n\in A_{>2}}\frac{t_{a_1}\ldots t_{a_n}}{n!}\int_{\sM_{\beta}} \prod_{i=1}^n\ch_2(\gamma_{a_i})\,.
\end{align}
Note that the $Q^\beta t_{a_1}\ldots t_{a_n}$ coefficient of $Z_{\textup{prim}}^{\PT}(Q, q, {\bf t})$ is the same as 
\[Z^{\PT}_\beta\big(q|\ch_2(\gamma_{a_1})\ldots \ch_2(\gamma_{a_n})\big)\,,\]
which is the expansion of a rational function with $q\leftrightarrow q^{-1}$ symmetry, as we have shown. 
The fact that we only consider $a_i\in A_{>2}$ is not losing any information, see Remark \ref{rmk: cohdegreeconstraints}. We emphasize that $Z_{\textup{prim}}^{\mathsf{M}}$ only contains information concerning curve classes with $d_{\beta}=1$. In particular, if $X$ is such that $-K_X\cdot \beta>1$ for every effective curve class $\beta$ (e.g. $X=\BP^3, X=\BP^1\times \BP^1\times \BP^1$ or $X$ is a cubic 3-fold),  then $Z_{\textup{prim}}^{\mathsf{M}}=0$. By Remark \ref{rmk: cohdegreeconstraints} the sum over $a_i$ in the definition of $Z_{\textup{prim}}^{\mathsf{M}}$ is quite simple: either $n=1$ and $a_1\in A_4$ or $n=2$ and $a_1, a_2\in A_3$.

\begin{theorem}\label{thm: primarywc}Let $X$ be a Fano 3-fold. We have the following equality of formal generating series:
    \begin{equation}
        \label{eq: todaexponential}
    Z_{\textup{prim}}^{\mathsf{PT}}(Q, q, \bt)=\exp\left(\frac{1}{q^{1/2}+q^{-1/2}}Z_{\textup{prim}}^{\mathsf{M}}(Q, \bt)\right)Z_{\textup{prim}}^{\mathsf{L}}(Q, q, \bt)\,.\end{equation}
\end{theorem}
\begin{remark}\label{rmk: positiveclassesprimary}
    More generally, if $X, \beta$ are as in the statement of Theorem \ref{thm: main} then the $Q^\beta$ coefficient of both sides matches.
\end{remark}
\begin{proof}
For notational convenience we will work with the polarization $H=-K_X$. Given a subset $I\subseteq [n]$ let us denote $D_I=\prod_{i\in I} \ch_2(\gamma_i)$. By the main wall-crossing formula and Lemma \ref{lem: bracketprimary} we have
\begin{align*}\nonumber\int_{\PT_{\beta, m}}D=\sum_{\substack{\beta_0+\ldots+\beta_k=\beta\\
d_{\beta_i}=1\textup{ for }i=1,\ldots, k\\
    m_0+\ldots+m_k=m\\
    0< m_1\leq m_2\leq \ldots\leq m_k}}
    \omord\Big(m_1,\ldots, m_k\Big)
    \sum_{I_0\sqcup\ldots\sqcup I_k=[n]}&\left(\int_{\sL_{\beta_0, m_0}}D_{I_0}\right)\\[-1cm]
    &\prod_{j=1}^k\left( (-1)^{m_j-1/2}\left(\int_{\sM_{\beta_j, m_j}}D_{I_j}\right)\right)\label{eq: wcprimaryI}
    \end{align*}
Note that, since $d_{\beta_i}=1$ is odd, the sum runs over half integers $m_1, \ldots, m_k\in \tfrac{1}{2}+\BZ$, and in particular they cannot be 0. The second sum is over the set of all possible partitions of $[n]$ into $k+1$ parts.

By Lemma \ref{lem: easychiindependence} the integral $\int_{\sM_{\beta_j, m_j}}D_{I_j}$ does not depend on $m_j$, and we simply write it as $\int_{\sM_{\beta_j}}D_{I_j}$. 
Observe that, given a $k$-tuple $(m_1, \ldots, m_k)$ of positive numbers, not necessarily ordered, the number of permutations that make it ordered is equal to $1/\omord(m_1, \ldots, m_k)$. Hence
\begin{align*}\sum_{\substack{m_j\in \frac{1}{2}+\BZ_{\geq 0}\\
m_1\leq m_2\leq \ldots\leq m_k}}\omord(m_1, \ldots, m_k)\prod_{j=1}^k (-1)^{m_j-1/2}q^{m_j}&=\sum_{m_j\in \frac{1}{2}+\BZ_{\geq 0}}\frac{1}{k!}\prod_{j=1}^k (-1)^{m_j-1/2}q^{m_j}\\
&=\frac{1}{k!}\left(\frac{1}{q^{1/2}+q^{-1/2}}\right)^k\,.
\end{align*}
Therefore we can rewrite the wall-crossing above as
\[Z_\beta^\PT(q, D)=\sum_{\substack{\beta_0+\ldots+\beta_k=\beta\\
d_{\beta_i}=1\textup{ for }i=1,\ldots, k}}\sum_{I_0\sqcup\ldots\sqcup I_k=[n]}Z_{\beta_0}^{\sL}(q| D_{I_0})\frac{1}{k!(q^{1/2}+q^{-1/2})^k}\prod_{j=1}^k \left(\int_{\sM_{\beta_j}}D_{I_j}\right)\,.\]
The right hand side is precisely the $Q^{\beta}t_{a_1}\ldots t_{a_n}$ coefficient of the right hand side of \eqref{eq: todaexponential}.\qedhere
\end{proof}

\subsection{Stable pairs/Gopakumar--Vafa and strong rationality}\label{subsec: primaryrationality}

Pandharipande conjectured in \cite{3questions} that, similarly to the Calabi--Yau case, the primary section of the Gromov--Witten ($\GW$) theory of 3-folds is governed by certain integers, often called Gopakumar--Vafa ($\GV$) invariants or BPS invariants. Together with the $\GW/\PT$ correspondence this predicts a very strong form of rationality, which we now review following \cite[Section 3.6]{PT}. 

While stable pairs are naturally a disconnected theory, in the sense that the support of a stable pair might be disconnected, one can formally define ``connected $\PT$ invariants'' by taking a logarithm of the partition function. For primary insertions $D$ we define $Z_{\beta}^{\PT, \circ}(q|D)$ as the coefficients of

\[\sum_{\beta> 0}Q^\beta \sum_{a_1, \ldots, a_n\in A_{>2}}\frac{t_{a_1}\ldots t_{a_n}}{n!}Z_{\beta}^{\PT, \circ}\big(q| \ch_2(\gamma_{a_1}\big)\ldots\ch_2(\gamma_{a_n})\big)=\log(Z_{\textup{prim}}^{\PT}(Q, q, {\bf t}))\,.\]
Note that $Z_{\beta}^{\PT, \circ}(q|D)$ is a polynomial in $Z_{\beta'}^{\PT}(q|D')$ for $\beta'\leq \beta$ and $D'=D_I$ for some $I\subseteq [n]$, and vice versa, so they carry the same information. In particular, Theorem \ref{thm: main} also implies that $Z_{\beta}^{\PT, \circ}(q|D)$ is also a rational function symmetric under $q\leftrightarrow q^{-1}$. The $\GV/\PT$ correspondence states that there exist integer numbers $n_{g,\beta}(\gamma_1,\ldots, \gamma_n)$ for $0\leq g\leq N_\beta$, where $N_\beta$ is some integer depending only on $\beta$, such that
\begin{equation}\label{eq: PT/GV}Z_{\beta}^{\PT, \circ}\big(q|\ch_2(\gamma_1)\ldots\ch_n(\gamma_n)\big)=\sum_{g=0}^{N_\beta} n_{g,\beta}(\gamma_1,\ldots,\gamma_n)(q^{1/2}+q^{-1/2})^{2g-2+d_\beta}\,.\end{equation}
When $X$ is Calabi--Yau a similar formula holds, with $d_\beta=0$, in which case the poles come from the $g=0$ GV invariant. In the Fano case one needs to distinguish between $d_\beta=1$, in which case we may still have a pole at $q=-1$ coming from genus 0 contributions, or $d_\beta>1$, in which case there are no poles other than $q=0$. We give now a new proof of this statement by proving Theorem \ref{thm: primarystrongrationality}.\footnote{Note that Theorem \ref{thm: primarystrongrationality} is weaker than the GV prediction. When $d_\beta>2$, \eqref{eq: PT/GV} implies that $Z_{\beta}^{\PT, \circ}\big(q|D)$ is divisible by $(q^{1/2}+q^{-1/2})^{d_\beta-2}$, which we do not know how to prove with our techniques.}

\begin{proof}[Proof of Theorem \ref{thm: primarystrongrationality}]
Taking a logarithm of Theorem \ref{thm: primarywc} gives 
\[Z_{\beta}^{\PT, \circ}(q|D)=\frac{1}{q^{1/2}+q^{-1/2}}\int_{\sM_\beta} D+Z_{\beta}^{\sL, \circ}(q|D)\]
where $Z_{\beta}^{\sL, \circ}$ is defined in the same way as $Z_{\beta}^{\PT, \circ}$. Since $Z_{\beta}^{\sL, \circ}(q| D)$ is a polynomial in $Z_{\beta'}^{\sL}(q| D')$ and the latter are Laurent polynomials in $q$ we conclude the proof.\qedhere 
\end{proof}

The following is an immediate consequence:

\begin{corollary}\label{cor: laurentpoly}
Let $X, \beta$ be as in the statement of Theorem \ref{thm: main} and let $D$ be a primary insertion. Then $Z_{\beta}^{\PT}(q|D)$ has poles only at $q=-1$ (and $q=0$). If $d_{\beta'}>1$ for every $0<\beta'\leq \beta$ then $Z_{\beta}^{\PT}(q|D)$ is a Laurent polynomial.
\end{corollary}

\begin{remark}
The Laurent polynomial phenomena can be seen for example in the calculations of stable pair invariants for $\BP^3$ (cf. \cite[Appendix A]{MOOP}) or the cubic 3-fold (cf. \cite[Theorem 21]{moreira}). More generally, a small modification of Lemma \ref{lem: bracketprimary} shows the following. Let $K=\min_{\beta \textup{ effective}}d_\beta$ (for example $K=4$ if $X=\BP^3$), and consider a non-primary descendent $D=\prod_{i=1}^n \ch_{2+k_i}(\gamma_i)$ such that $\sum_{i=1}^{n}k_i\leq K-2$. Then $Z_{\beta}^{\PT}(q|D)$ agrees again with $Z_{\beta}^{\sL}(q|D)$, and hence it is a Laurent polynomial. 
\end{remark}

\begin{remark}
    Maulik and Ranganathan show in \cite[Theorem B]{MRgwpttoric} that if $X$ is a toric 3-fold such that every toric divisor is nef then the generating series $Z_{\beta}^\PT(q|D)$ are Laurent polynomials for any primary $D$ and any $\beta$. Corollary \ref{cor: laurentpoly} reproves their result (in the empty boundary case) since $d_\beta \geq 2$ for any effective curve class $\beta$. The last statement follows from the fact that $-K_X$ is the sum of the toric divisors corresponding to faces of the moment polytope of $X$ and that the cone of effective curve classes is generated by toric curves corresponding to edges of the polytope \cite[Proposition 1.6]{reid}.
\end{remark}

\end{document}